\documentclass[11pt,a4paper]{amsart}

\usepackage[a4paper,margin=1in]{geometry}
\usepackage[T1]{fontenc}
\usepackage[utf8]{inputenc}
\usepackage{lmodern}
\usepackage{amssymb,mathtools,bm}
\usepackage{bbm}
\allowdisplaybreaks
\usepackage{graphicx,xcolor}
\usepackage[colorlinks=true,linkcolor=blue,citecolor=blue,urlcolor=blue]{hyperref}
\usepackage{multirow}
\usepackage{array} 
\usepackage{booktabs}

\usepackage{setspace}
\usepackage{amsthm}
\newtheorem{theorem}{Theorem}[section]
\newtheorem{lemma}[theorem]{Lemma}

\newtheorem{proposition}[theorem]{Proposition}
\theoremstyle{remark}
\newtheorem{remark}{Remark}

\usepackage[
    backend=biber,
    style=numeric,
    citestyle=numeric,
    sorting=nyt,
    maxnames=6,
    giveninits=true,
    doi=false,
    isbn=false,
    url=false,
    eprint=false
]{biblatex}

\DeclareNameAlias{sortname}{family-given}
\DeclareNameAlias{default}{family-given}

\newbibmacro*{author+year}{%
  \printnames{author}%
  \setunit{\space}%
  \printtext[parens]{\printfield{year}}}

\DeclareBibliographyDriver{article}{%
  \usebibmacro{author+year}%
  \newunit\newblock
  \printfield[titlecase]{title}%
  \newunit\newblock
  \printfield[journaltitle]{journaltitle}%
  \setunit*{\addspace}%
  \printfield{volume}%
  \iffieldundef{number}{}{(\printfield{number})}%
  \addcolon\space%
  \printfield{pages}%
  \finentry}

\renewbibmacro{in:}{}

\date{\today}

\title[Adaptive Thresholding Variance Estimation on $\mathbb{S}^d$]{Adaptive Thresholding for Wavelet-Based Nonparametric Heteroskedastic Variance Estimation on the Sphere}

\author[C. Durastanti]{Claudio Durastanti}
\address[C. Durastanti]{Department of Basic and Applied Sciences for Engineering, Sapienza University of Rome.}
\email{claudio.durastanti@unipd.it}

\author[R. Shevchenko]{Radomyra Shevchenko}
\address[R. Shevchenko]{J. A. Dieudonn\'e Laboratory - C\^ote d'Azur University, and Centrale Méditerranée, Nice.}
\email{radomyra.shevchenko@univ-cotedazur.fr}

\date{\today}

\begin{document}

%%%%%% Abstract %%%%%%
\begin{abstract}
This paper addresses the nonparametric estimation of a spatially varying, heteroskedastic variance function on the unit sphere within a regression framework. While adaptive regression estimation is well-established on manifolds, characterizing localized noise structures presents unique theoretical obstacles due to bias propagation from the unknown mean function. To circumvent this, we propose a fully data-driven, multiresolution estimator based on localized spherical frames, namely, needlets, combined with a hard-thresholding protocol and a sample-splitting scheme. The approach exploits the excellent spatial and frequency localization properties of needlets to adaptively capture the local features of the variance function. We prove that the proposed estimator achieves the minimax-optimal rate of convergence over spherical Besov spaces under standard loss functions, exhibiting spatial adaptivity without requiring prior knowledge of the regularity of the variance function. This adaptivity highlights the efficacy of the method in analyzing spherical data characterized by complex heteroskedastic errors, with potential applications in fields such as cosmology, environmental modeling, and geophysics.\\
\textbf{Keywords:} Nonparametric regression, Heteroskedasticity, Variance function estimation, Spherical data, Needlets, Minimax rate, Besov spaces.\\
\textbf{2010 MSC:} 62G08, 62G20, 65T60
\end{abstract}

%%%%%% Main Text %%%%%%
\maketitle
\section{Introduction}\label{sec:intro}

In many modern applications, regression models are naturally defined on directional and manifold domains rather than in classical Euclidean spaces. Relevant examples include the analysis of cosmic microwave background radiation in cosmology~\cite{gerbinoetal}, global climate and environmental monitoring~\cite{cortereal}, and the modeling of geophysical or oceanographic processes~\cite{ckrl00}. In these complex spatial settings, the observational noise typically varies across locations, inducing a heteroskedastic regression structure. While nonparametric methods for adaptive regression function estimation are well-established on manifolds, the problem of estimating a spatially varying, heteroskedastic variance function has received comparatively little attention in non-Euclidean contexts. Consequently, most existing frameworks rely on the assumption of homoskedastic errors, which substantially simplifies the theoretical analysis but is rarely satisfied in practice.

Variance function estimation in nonparametric regression has a long history in the Euclidean literature. Early contributions include difference-based constructions~\cite{rice84} and direct residual-based procedures~\cite{mull87}. The fundamental effect of estimating an unknown mean function on the subsequent variance estimation was systematically studied by~\cite{hc89}, while related quadratic-form methods were analyzed by~\cite{dmw98}. Although difference-based estimators can circumvent explicit mean estimation, their performance remains highly sensitive to the experimental design and the underlying regularity of the regression function~\cite{bl07,mbwf05}. Subsequent adaptive and minimax analyses clarified how the joint smoothness of the mean and variance functions governs the attainable convergence rates~\cite{CW08,wbc08,sgwh20}. These results establish the Euclidean benchmark for the non-Euclidean problem considered here. Related work has also expanded into minimax-optimal testing for heteroskedasticity~\cite{kk25}.

To capture localized features on the sphere and related manifolds, multiscale frame methods, most notably needlets, have become an important paradigm. In density estimation,~\cite{BKMP09} introduced adaptive needlet thresholding procedures that achieve minimax-optimal rates under $L^p$ losses. Extensions to regression models and spin spaces were subsequently developed by~\cite{DGM12,Mon11}, while~\cite{Dur16} established near-optimal global thresholding strategies. Further developments include adaptive estimation of derivatives on the torus~\cite{dt23}, applications to random coefficient binary choice models~\cite{gp18}, and needlet-based uncertainty quantification via asymptotically valid confidence bands~\cite{knp11}. 

This multiscale framework is particularly crucial on directional domains because standard difference-based constructions, which rely on ordered covariates or local Cartesian neighborhoods, do not transfer directly to a rotationally invariant spherical design. While geodesic or graph-based alternatives can be conceived, they do not naturally provide the simultaneous multiscale localization and Besov adaptivity supplied by needlets. By contrast, needlets provide localized multiscale representations and a canonical route to adaptation over spherical Besov classes.

Recent years have witnessed a surge in nonparametric methodology for directional and manifold-valued data, including regression on nonstandard spaces and geometric statistical methods (see, among others,~\cite{cm24,schotz22}).  Often, the estimation is based on some generalizations of kernel-based procedures approached via the Fréchet mean (as, for instance, in \cite{PetersenMueller, im2025local}) or via the tangent space or higher order approximations (see, for example, \cite{wiem2026nonparametric,pelletier2006non}). For instance, a recent approach to regression on the sphere combines spectral basis decompositions with rotations (cf.~\cite{jeon2024density}). However, these developments focus almost exclusively on the conditional mean or density, leaving the question of variance estimation in heteroskedastic settings unaddressed.

To the best of our knowledge, the adaptive nonparametric estimation of a spatially varying variance function has not previously been developed on non-Euclidean domains. This paper bridges this substantial methodological gap. While the core challenge in this setup lies in designing a valid de-biasing protocol that preserves minimax optimality when the mean is unknown, we solve this problem for the unit sphere $\mathbb{S}^d$ in arbitrary dimension $d \ge 2$. We propose a fully data-driven, sample-splitting needlet procedure that first recovers the regression function on independent subsamples, and subsequently constructs bias-corrected quadratic quantities to estimate the variance function. Adaptive hard thresholding is then applied to the resulting needlet coefficients. We establish adaptive $L^p$-risk bounds over spherical Besov classes and prove that they are minimax-optimal throughout the parameter space, except in a minor boundary regime.

More precisely, we consider observational pairs $(X_i,Y_i)$, $i=1,\ldots,N$, where the covariates $X_i$ lie on the unit sphere $\mathbb{S}^d$, and the responses $Y_i$ are generated by the heteroskedastic regression model
\begin{equation}\label{eq:model}
Y_i = g(X_i) + \sigma(X_i)\varepsilon_i, \qquad i=1,\ldots,N,
\end{equation}
where $g:\mathbb S^d\to\mathbb R$ is the unknown regression function, $\sigma:\mathbb S^d\to(0,\infty)$ is the unknown scale function, and $\{\varepsilon_i\}$ are independent, mean-zero, sub-Gaussian errors with unit variance. The sub-Gaussian assumption allows us to leverage sharp deviation bounds for the needlet coefficients; furthermore, since squared sub-Gaussian variables are sub-exponential, it enables explicit control of higher-order moments during the variance estimation steps. The target of our inference is the variance function
\begin{equation*}\label{eqn:varscale}
V(x)=\sigma(x)^2, \qquad x \in \mathbb{S}^d.
\end{equation*}
When $V$ varies across the sphere, estimation procedures calibrated under a homoskedastic working model may retain consistency for the regression function, but they become statistically inefficient and lead to miscalibrated uncertainty quantification. In particular, because the covariance structure of the empirical needlet coefficients depends on the local noise level, a constant-variance approximation distorts standard error estimates and threshold choices. Estimating $V$ is therefore essential for both adaptive function reconstruction and valid statistical inference.

The $L^p$-risk provides the natural metric of accuracy in this framework, quantifying the expected $L^p$-norm of the error between an estimator $\widehat{V}_N$ and the true variance function $V$. Our minimax analysis characterizes how this risk decreases with $N$ while simultaneously adapting to the unknown, potentially different smoothness degrees of $g$ and $V$. We show that the proposed estimator attains the minimax risk over the regularity class up to multiplicative constants, satisfying the standard adaptive optimality benchmarks established in the literature~\cite{donoho98,tsyb09,WASA}.

Although our methodology and theoretical analysis are implemented on the unit sphere $\mathbb S^d$, the mathematical architecture of our results is applicable beyond this specific manifold. The fundamental components underlying our construction—namely, the spectral decomposition of the Laplace--Beltrami operator, the existence of well-localized multiscale frames, and compatible localized cubature formulas—are available on a broad class of compact Riemannian manifolds (see, for instance, ~\cite{knp11} and the references therein). 

We focus on the sphere here primarily for two reasons. First, spherical needlets admit an explicit construction via spherical harmonics, leading to sharper constants, highly transparent proofs, and a streamlined presentation. Second, the sphere represents the natural domain for the aforementioned applications in cosmology and geophysics, where the geometry is intrinsic to the data generation process. The insights developed here provide a clear blueprint for extensions to more general compact manifolds.

\subsubsection*{Plan of the paper}
The remainder of the paper is organized as follows. Section~\ref{sec:back} introduces the necessary background on harmonic and needlet analysis on the sphere, alongside the formal definition of spherical Besov spaces. Section~\ref{sec:framework} formalizes the heteroskedastic regression framework, details the sample-splitting needlet estimation protocol, and establishes key auxiliary lemmas. Section~\ref{sec:stocresults} presents our main theoretical contributions concerning the minimax rates. Section \ref{sec:numerics} collects some numerical illustrations. Finally, Section~\ref{sec:proofs} collects the proofs of all the theoretical results.

\section{Background and preliminary results}\label{sec:back}
In this section, we collect essential background material on harmonic and needlet analysis on the d-dimensional unit sphere, alongside the functional characterization of spherical Besov spaces. 
Section~\ref{sec:needlets} reviews the algebraic construction and primary localization properties of spherical needlet frames, drawing primarily on~\cite{MP11,NPW06a,NPW06b}. 
Section~\ref{sec:besov} formally introduces Besov spaces on $\mathbb{S}^d$ and summarizes their fundamental approximation and topological embedding properties, following~\cite{bcd11,BKMP09,WASA}.

\subsubsection*{Notation} 
Let $x$ denote a point on the unit sphere $\mathbb{S}^d$, and let $\mathrm{d}x$ be the uniform Lebesgue surface measure on $\mathbb{S}^d$ scaled such that $\int_{\mathbb{S}^d} \mathrm{d}x = 1$. The total unnormalized surface area of the sphere is denoted by $\omega_d := 2\pi^{(d+1)/2}/\Gamma((d+1)/2)$. For a measurable function $f:\mathbb{S}^d \to \mathbb{R}$, its $L^p$-norm is defined as
\[
\|f\|_p = \|f\|_{L^p(\mathbb{S}^d)} = \left( \int_{\mathbb{S}^d} |f(x)|^p\,\mathrm{d}x \right)^{1/p}
\]
for $1 \le p < \infty$, with the standard adaptation $\|f\|_\infty = \operatorname{ess\,sup}_{x \in \mathbb{S}^d} |f(x)|$ for $p = \infty$. In particular, the space $L^2(\mathbb{S}^d)$ is the Hilbert space of square-integrable functions with respect to the normalized measure $\mathrm{d}x$, equipped with the inner product $\langle f, g \rangle = \int_{\mathbb{S}^d} f(x)g(x)\,\mathrm{d}x$. Throughout the paper, the notation $a \lesssim b$ indicates that $a \le C b$ for some positive constant $C$ that depends at most on the structural parameters of the model (such as $d$, $p$, and the Besov regularities) but is strictly independent of the sample size $N$ and the resolution levels. We write $a \simeq b$ when both $a \lesssim b$ and $b \lesssim a$ hold simultaneously.

\subsection{Needlet frames on the sphere and their properties}\label{sec:needlets}
We recall the algebraic construction of spherical needlets and summarize their fundamental properties; see~\cite{NPW06a,NPW06b} for comprehensive details. 
Let $\mathcal{H}_\ell(\mathbb{S}^d)$ denote the space of spherical harmonics of degree $\ell$ on the $d$-dimensional unit sphere $\mathbb{S}^d \subset \mathbb{R}^{d+1}$ for $d \ge 2$. The dimension of this subspace is given by
\[
Z_{d,\ell} = \frac{2\ell+d-1}{\ell+d-1} \binom{\ell+d-1}{\ell} \simeq \ell^{d-1} \quad \text{as } \ell \to \infty.
\]

For each degree $\ell \ge 0$, we choose an orthonormal basis of real-valued spherical harmonics $\{ Y_{\ell,m} : m = 1, \ldots, Z_{d,\ell} \}$ for $\mathcal{H}_\ell(\mathbb{S}^d)$. The collection of all such harmonics across $\ell \ge 0$ constitutes a complete orthonormal basis of $L^2(\mathbb{S}^d)$. The spaces $\mathcal{H}_\ell(\mathbb{S}^d)$ are mutually orthogonal, and any square-integrable function $f \in L^2(\mathbb{S}^d)$ admits the orthogonal harmonic expansion
\[
f(x) = \sum_{\ell \ge 0} P_\ell f(x) = \sum_{\ell \ge 0} \sum_{m=1}^{Z_{d,\ell}} a_{\ell,m} Y_{\ell,m}(x),
\]
where $a_{\ell,m} = \int_{\mathbb{S}^d} f(x) Y_{\ell,m}(x)\,\mathrm{d}x$, and $P_\ell: L^2(\mathbb{S}^d) \to \mathcal{H}_\ell(\mathbb{S}^d)$ represents the orthogonal projection operator. By the addition theorem for spherical harmonics, the reproducing kernel of $\mathcal{H}_\ell(\mathbb{S}^d)$ can be written intrinsically as
\[
\sum_{m=1}^{Z_{d,\ell}} Y_{\ell,m}(x) Y_{\ell,m}(y) = \frac{Z_{d,\ell}}{\omega_d} \frac{C_\ell^{\left(\frac{d-1}{2}\right)}(\langle x, y \rangle)}{C_\ell^{\left(\frac{d-1}{2}\right)}(1)},
\]
where $\langle x, y \rangle$ is the standard Euclidean inner product in $\mathbb{R}^{d+1}$, $\omega_d$ is the surface area of $\mathbb{S}^d$ defined in Section~\ref{sec:back}, and $C_\ell^{(\lambda)}$ denotes the Gegenbauer polynomial of degree $\ell$ with parameter $\lambda = (d-1)/2$ (see~\cite{MP11}).

Fix a scaling parameter $B>1$. Following the Littlewood–Paley decomposition framework, there exist a set of cubature points and positive weights $\{ (\xi_{j,k}, \lambda_{j,k}) : j\ge1,\, k=1,\ldots,K_j \}$ that enable exact quadrature for spherical polynomials up to degree $\lfloor 2B^{j+1} \rfloor$. The total number of cubature points at resolution level $j$ scales quadratically with the dimension, satisfying $K_j \simeq B^{dj}$, while the weights satisfy $\lambda_{j,k} \simeq B^{-dj}$ uniformly. Each pair $(\xi_{j,k}, \lambda_{j,k})$ can be interpreted geometrically as a pixel centered at $\xi_{j,k}$ with an associated area proportional to $\lambda_{j,k}$.

The spherical needlets are then defined by
\[
\psi_{j,k}(x) = \sqrt{\lambda_{j,k}} \sum_{\ell \in \Lambda_j} b\!\left(\frac{\ell}{B^j}\right) \frac{Z_{d,\ell}}{\omega_d} \frac{C_\ell^{\left(\frac{d-1}{2}\right)}(\langle x, \xi_{j,k} \rangle)}{C_\ell^{\left(\frac{d-1}{2}\right)}(1)}, \qquad x \in \mathbb{S}^d,
\]
or equivalently via the spherical harmonics basis as
\[
\psi_{j,k}(x) = \sqrt{\lambda_{j,k}} \sum_{\ell \in \Lambda_j} b\!\left(\frac{\ell}{B^j}\right) \sum_{m=1}^{Z_{d,\ell}} Y_{\ell,m}(\xi_{j,k}) Y_{\ell,m}(x),
\]
where $\Lambda_j = \{\ell \in \mathbb{N} : B^{j-1} \le \ell \le B^{j+1}\}$ and $b:\mathbb{R}\to\mathbb{R}$ is a smooth window function satisfying:
\begin{enumerate}
  \item $b$ has compact support in $[B^{-1},B]$;
  \item $b \in C^\infty(\mathbb{R})$;
  \item the partition of unity property: $\sum_{j\ge 0} b^2\left(\frac{\ell}{B^j}\right) = 1$ for all $\ell \ge 1$.
\end{enumerate}

Needlets enjoy a crucial double localization property: they are simultaneously localized in both the spatial and frequency domains. In the frequency domain, Property (1) ensures that each needlet is constructed as a localized band-limited combination of multipole components. As a consequence of Property (2) and the regularity of $b$, for every arbitrarily large $M \in \mathbb{N}$, there exists a constant $C_M>0$ such that
\[
|\psi_{j,k}(x)| \le \frac{C_M B^{j\frac{d}{2}}}{\left(1 + B^j d_{\mathbb{S}^d}(x,\xi_{j,k})\right)^M},
\]
where $d_{\mathbb{S}^d}$ is the geodesic distance on the sphere. To ensure the integrability of the tails across $\mathbb{S}^d$, $M$ is typically chosen such that $M > d$. This sharp spatial localization guarantees that needlet coefficients extract local high-frequency details, a feature essential for adaptive variance function estimation under spatial heteroskedasticity.

Consequently, for any $p \in [1,\infty]$, the $L^p$-norm of a needlet scales explicitly with the dimension $d$, satisfying
\begin{equation}\label{eqn:norm}
c_{\psi,p} B^{jd\left(\frac{1}{2} - \frac{1}{p}\right)} \le \|\psi_{j,k}\|_{L^p(\mathbb{S}^d)} \le C_{\psi,p} B^{jd\left(\frac{1}{2} - \frac{1}{p}\right)}
\end{equation}
where $0 < c_{\psi,p} \le C_{\psi,p} < \infty$ are constants independent of $j$ and $k$.

Recall that a countable collection of functions $\{ e_i \}_{i\ge 0}$ forms a tight frame for $L^2(\mathbb{S}^d)$ if 
$\sum_{i\ge0} |\langle f,e_i\rangle|^2 = c\| f \|_{L^2(\mathbb{S}^d)}^2$ for all $f \in L^2(\mathbb{S}^d)$. The spherical needlets constructed above form a tight frame, with $c=1$ allowing for perfect reconstruction without numerical inversion. Finally, by using Property (3), every $f \in L^2(\mathbb{S}^d)$ admits the multiresolution expansion
\begin{equation}\label{eqn:needf}
f(x) = \sum_{j\ge 1} \sum_{k=1}^{K_j} f_{j,k}\, \psi_{j,k}(x), \qquad f_{j,k} = \int_{\mathbb{S}^d} f(x)\, \psi_{j,k}(x)\,\mathrm{d}x.
\end{equation}
The coefficients $\{f_{j,k}:j \geq 1,\, k=1,\ldots, K_j\}$ thus characterize localized multiscale properties on the sphere, ensuring stable functional reconstructions since truncated needlet expansions preserve structural features under local regularity variations.

\subsection{Besov spaces on the sphere}\label{sec:besov}
We recall the formal definition of spherical Besov spaces and their equivalent characterization in terms of needlet coefficients. For rigorous constructions, structural properties, and detailed proofs, we refer the reader to~\cite{bcd11,BKMP09,WASA}. Let $f \in L^r(\mathbb{S}^d)$ admit the needlet tight-frame expansion given by \eqref{eqn:needf}.

Besov spaces on the $d$-dimensional sphere can be intrinsically defined via the approximation errors attained by suitable classes of smooth functions. Given a nested scale of functional subspaces $\mathcal{G}_t \subset L^r(\mathbb{S}^d)$, indexed by a continuous parameter $t \ge 0$ (such as the spaces of spherical polynomials of degree at most $\lfloor t \rfloor$), we define the best approximation error of $f$ as
\[
G_t(f;r) = \inf_{h \in \mathcal{G}_t} \| f - h \|_{L^r(\mathbb{S}^d)}.
\]
For a fixed dilation parameter $B > 1$, a function $f$ belongs to the spherical Besov space $B^s_{r,q}(\mathbb{S}^d)$ with smoothness $s > 0$, spatial integrability $1 \le r \le \infty$, and inter-scale summability $1 \le q \le \infty$ if the following dyadic sequence condition holds:
\[
\sum_{j \ge 0} B^{j s q} [G_{B^j}(f;r)]^q < \infty.
\]

In the needlet framework, this analytical condition can be mapped onto the decay properties of the frame coefficients. Specifically, a function $f \in L^r\left(\mathbb{S}^d\right)$ belongs to $B^s_{r,q}(\mathbb{S}^d)$ if and only if the sequence $w_j(f) := B^{js} \left[\sum_{k=1}^{K_j}\left( |f_{j,k}| \|\psi_{j,k} \|_{L^r(\mathbb{S}^d)} \right)^r\right]^{1/r}$ belongs to the sequence space $\ell^q(\mathbb{N}_0)$. Consequently, by inserting the norm scaling established in \eqref{eqn:norm}, the Besov space $B^s_{r,q}(\mathbb{S}^d)$ is equipped with the equivalent norm
\[
\| f \|_{B^{s}_{r,q}\left(\mathbb{S}^d\right)} = \| f \|_{L^r(\mathbb{S}^d)} + \left[ \sum_{j \ge 0} B^{j q \left( s + \frac{d}{2} - \frac{d}{r} \right)} \left( \sum_{k=1}^{K_j} |f_{j,k}|^r \right)^{\frac{q}{r}} \right]^{\frac{1}{q}} < \infty.
\]
The hyperparameters carry a clear statistical interpretation: $s$ measures the global smoothness, governing the power-law decay of the coefficients across resolution levels $j$; $r$ quantifies spatial integrability, regulating local spatial variations and burstiness at a fixed scale; and $q$ fine-tunes the inter-scale summability.

Intuitively, $B^s_{r,q}(\mathbb{S}^d)$ consists of functions whose $s$-th fractional derivatives are bounded in an $L^r$ sense. Spherical Besov spaces interpolate between classical Sobolev and H\"older spaces, providing a flexible framework to characterize both spatial smoothness and sparsity. For instance, when $r=q=2$, $B^s_{2,2}\left(\mathbb{S}^d\right) = H^s(\mathbb{S}^d)$ coincides with the standard fractional Sobolev space on the sphere, whereas the classical s-H\"older spaces correspond to $C^s(\mathbb{S}^d) = B^{s}_{\infty,\infty}(\mathbb{S}^d)$ for any non-integer $s>0$.

\subsubsection*{Besov balls} For a fixed radius $R>0$, the Besov ball of radius $R$ in $B^{s}_{r,q}\left(\mathbb{S}^d\right)$ is defined as
\[
B^{s}_{r,q}(R) = \left\{ f \in B^{s}_{r,q}(\mathbb{S}^d) \;:\; \|f\|_{B^{s}_{r,q}(\mathbb{S}^d)} \le R \right\}.
\]
These subsets provide a natural scale of compact or bounded function classes indexed by regularity, serving as the standard functional parameter spaces for minimax-optimal and adaptive estimation theory. When no ambiguity arises, we shall write simply $B^{s}_{r,q}$ or $B^{s}_{r,q}(R)$ without explicitly indicating the domain $\mathbb{S}^d$.

\subsubsection*{Norm embeddings and coefficient inequalities.} Spherical Besov spaces satisfy standard topological embedding properties (see \cite{BKMP09}). For any $1 \le r_1 \le r_2 \le \infty$, we have the continuous embeddings
\[
B^{s}_{r_2,q}(\mathbb{S}^d) \subset B^{s}_{r_1,q}(\mathbb{S}^d), \qquad B^{s}_{r_1,q}(\mathbb{S}^d) \subset B^{s - d\left(\frac{1}{r_1}-\frac{1}{r_2}\right)}_{r_2,q}(\mathbb{S}^d).
\]
In terms of the underlying needlet coefficients, these embeddings manifest at each resolution scale $j$ via H\"older's inequality as

\begin{equation}\label{eqn:embedding}
\sum_{k=1}^{K_j} |f_{j,k}|^{r_2} \le \sum_{k=1}^{K_j} |f_{j,k}|^{r_1}, \qquad \sum_{k=1}^{K_j} |f_{j,k}|^{r_1} \le K_j^{1-\frac{r_1}{r_2}} \sum_{k=1}^{K_j} |f_{j,k}|^{r_2}.
\end{equation}

\subsubsection*{Stability under multiplication.} 
Under the condition $s > d/r$, the space $B^{s}_{r,q}(\mathbb{S}^d)$ becomes a Banach algebra under pointwise multiplication. Specifically, if $f_1, f_2 \in L^\infty(\mathbb{S}^d) \cap B^s_{r,q}\left(\mathbb{S}^d\right)$, adapting the Euclidean product estimates (see~\cite[Corollary~2.86]{bcd11}) via localization arguments yields $f_1 f_2 \in B^{s}_{r,q}(\mathbb{S}^d)$. In particular,
\begin{equation*}\label{eqn:quadraticbesov} 
f \in B^{s}_{r,q}(\mathbb{S}^d) \implies f^2 \in B^{s}_{r,q}(\mathbb{S}^d).
\end{equation*}
This algebra property is crucial for the heteroskedastic regression model considered in this paper. Since nonparametric variance function estimation inherently relies on quadratic operations applied to residuals, the Besov algebra structure ensures that the underlying smoothness of the mean function is preserved under squaring, preventing unexpected regularizing or roughening effects.

\subsubsection*{Jackson-type inequalities.} Functions embedded in Besov spaces enjoy precise approximation properties under truncated needlet expansions. If $f \in B^s_{r,q}(\mathbb{S}^d)$ with $s > 0$, Jackson-type inequalities guarantee that the high-frequency residual tail beyond a cut-off resolution level $J$ is controlled exponentially by $s$. Formally, the $L^r$-norm of the tail satisfies
\begin{equation*}\label{eqn:jackson}
\left\| \sum_{j \ge J} \sum_{k=1}^{K_j} f_{j,k} \psi_{j,k} \right\|_{L^r(\mathbb{S}^d)} \le C_{s,r,\psi} B^{-J s} \| f \|_{B^s_{r,q}(\mathbb{S}^d)},
\end{equation*}
where the constant $C_{s,r,\psi} > 0$ is independent of both $f$ and the truncation level $J$ (see~\cite[Theorem 6.2]{NPW06b}).

\subsubsection*{Mixed-norm Jackson inequalities.} By blending the spatial localization of needlets with Nikol'skii-type inequalities on compact domains, we can bound the high-frequency tail across differing $L^p$ scales. If $p > r$, moving to a finer integration space introduces a structural dimensionality penalty due to spatial compression. Specifically, provided that $s > d\left(\frac{1}{r}-\frac{1}{p}\right)$, summing the resulting geometric series yields:
\[
\left\| \sum_{j\ge J} \sum_{k=1}^{K_j} f_{j,k} \psi_{j,k} \right\|_{L^p(\mathbb{S}^d)} \lesssim B^{-J\big(s-d(\frac{1}{r}-\frac{1}{p})\big)}\|f\|_{B^s_{r,q}}.
\]
Conversely, if $p \le r$, standard spatial domain inclusions yield the unpenalized bound:
\[
\left\| \sum_{j\ge J} \sum_{k=1}^{K_j} f_{j,k} \psi_{j,k}\right\|_{L^p(\mathbb{S}^d)} \lesssim B^{-J s}\left\|f\right\|_{B^s_{r,q}}.
\]
These mixed-norm bounds provide the analytical foundation required to establish sharp, spatially adaptive risk upper bounds for our variance estimator.

\section{Statistical framework and estimation procedure}\label{sec:framework}
In this section, we formalize the nonparametric regression model on the $d$-dimensional unit sphere under spatially varying, heteroskedastic noise conditions, and characterize its structure within the needlet domain. We then introduce our multi-stage estimation strategy designed to recover both the unknown conditional mean and variance functions. Our presentation highlights how the localized multiresolution architecture of needlets captures spatial variations across different scales, and details the de-biasing protocol required to circumvent non-trivial bias propagation under unknown directional design distributions and heteroskedastic errors.

\subsection{The stochastic model}\label{sec:model}
We now specify the stochastic and geometric assumptions underlying our estimation framework. The observations are modeled as noisy samples from a random field on the $d$-dimensional sphere, where the target mean and variance functions are assumed to lie within the Besov balls introduced in Section~\ref{sec:back}. 

Formally, we consider $N$ independent and identically distributed observational pairs \[\left\{(X_i, Y_i):i=1,\ldots, N \right\},\]  satisfying the nonparametric regression model ~\eqref{eq:model} on the sphere with spatially varying heteroskedastic noise:
\begin{equation*}
    Y_i = g(X_i) + \sigma(X_i)\varepsilon_i, \quad i = 1,\ldots,N,
\end{equation*}
where the model components comply with the following structural conditions:
\begin{itemize}
    \item The covariates $\{X_i\}_{i=1}^N$ are i.i.d.\ random variables distributed uniformly on $\mathbb{S}^d$. Their common probability law is given directly by $\mathbb{P}(X_i \in \mathrm{d}x) = \mathrm{d}x$.
    \item The unknown regression function $g : \mathbb{S}^d \to \mathbb{R}$ represents the conditional mean $g(x) = \mathbb{E}[Y_i \mid X_i = x]$ and is globally bounded, with $\|g\|_\infty \le G < \infty$.
    \item The standard deviation function $\sigma : \mathbb{S}^d \to (0,\infty)$ dictates the local noise amplitude and is bounded from above and below, such that $0 < \sigma_{\min} \le \sigma(x) \le S < \infty$ for all $x \in \mathbb{S}^d$.
    \item The errors $\{\varepsilon_i\}_{i=1}^N$ are i.i.d.\ random variables independent of $\{X_i\}_{i=1}^N$. They are centered with unit variance, $\mathbb{E}[\varepsilon_i] = 0$, $\mathbb{E}[\varepsilon_i^2] = 1$, and exhibit sub-Gaussian tails. Formally, there exists a constant $\kappa > 0$ such that $\mathbb{E}[\exp(t \varepsilon_i)] \le \exp(\kappa^2 t^2 / 2)$ for all $t \in \mathbb{R}$. This implies that all higher-order absolute moments $\gamma_m = \mathbb{E}[|\varepsilon_i|^m]$ are finite and well-controlled for $m \ge 3$.
\end{itemize}
The target of our statistical inference is the conditional variance function $V(x) = \operatorname{Var}(Y_i \mid X_i = x) = \sigma^2(x)$ for $x \in \mathbb{S}^d$. For presentation convenience, the sample size $N$ is assumed to be even.

\subsubsection*{Needlet decomposition.} Utilizing the tight frame properties of the spherical needlet system discussed in Section~\ref{sec:needlets}, both the mean function $g$ and the variance function $V$ admit stable, localized multiresolution expansions:
\begin{equation*}\label{eqn:coeff}
\begin{split}
    g(x) &= \sum_{j\ge 0}\sum_{k=1}^{K_j} g_{j,k}\,\psi_{j,k}(x), \qquad g_{j,k} = \int_{\mathbb{S}^d} g(x)\psi_{j,k}(x)\,\mathrm{d}x,\\
    V(x) &= \sum_{j\ge 0}\sum_{k=1}^{K_j} v_{j,k}\,\psi_{j,k}(x), \qquad v_{j,k} = \int_{\mathbb{S}^d} V(x)\psi_{j,k}(x)\,\mathrm{d}x.
\end{split}
\end{equation*}
The recovery of $V(x)$ is therefore equivalent to estimating the sequence of its continuous needlet coefficients $\{v_{j,k}\}$ from the discrete dataset.

\subsubsection*{Heteroskedastic structure.} The scale function $\sigma(\cdot)$ characterizes the spatial inhomogeneity of the noise across the directional domain. This heteroskedastic profile inherently alters the local statistical properties of the model: regions characterized by high noise energy require a more conservative thresholding regularization to suppress spurious stochastic fluctuations, whereas regions embedded in low-noise environments allow for finer spatial adaptations.

By mapping the variance function $V(x)$ onto its multiscale coefficients $\{v_{j,k}\}$, we obtain a simultaneous description of this spatial variation in both frequency and space. Low-frequency coefficients, corresponding to small resolution levels $j$, capture macroscopic, global trends in the noise intensity, while high-frequency coefficients, that is, large $j$, represent highly localized anomalies and rapid spatial transitions. This double localization is particularly convenient for establishing minimax data-driven procedures, as it allows for a pixel-by-pixel and scale-by-scale thresholding protocol that responds dynamically to the local empirical signal-to-noise ratio. These coefficients form the mathematical foundation for the adaptive estimation architecture detailed in the next subsection.

\subsection{Needlet-based estimation procedure}\label{sec:est}
We construct adaptive estimators for the conditional mean and variance functions $g$ and $V=\sigma^2$ in the heteroskedastic spherical regression framework by leveraging the simultaneous space–frequency localization of spherical needlets. This multi-stage procedure yields fully data-driven, spatially adaptive estimators capable of accommodating localized noise patterns while preserving the structural features of the target functions across varied smoothness regimes.

The estimation architecture relies on embedding properties within spherical Besov balls. Formally, we assume that the conditional mean and variance functions satisfy
\[
g \in B^{\alpha}_{\rho,q'}(R_g), \qquad V \in B^{\beta}_{\mu,q''}(R_V),
\]
where $B^s_{r,q}(R)$ denotes the Besov ball defined in Section~\ref{sec:besov}. To guarantee the continuity and pointwise bounded evaluation of these functional targets on $\mathbb{S}^d$, we require the baseline regularities to satisfy $\alpha > d/\rho$ and $\beta > d/\mu$.

To implement our variance recovery protocol, we introduce the second-moment compound function
\[
h(x) = g(x)^2 + V(x), \qquad x \in \mathbb{S}^d,
\]
which inherits its analytical smoothness from the joint regularities of $g$ and $V$. By virtue of the multiplication algebra property of spherical Besov spaces under the stated constraints, it follows that $h \in B^{s_0}_{r_0,q_0}(R_h)$, where the effective parameters are given by
\[
s_0 = \min(\alpha,\beta), \qquad r_0 = \min(\rho,\mu),
\]
while $q_0 = q'$ if $s_0=\alpha < \beta$, $q_0 = q''$ if $s_0=\beta < \alpha$, and $q_0=\max(q', q'')$ when $\alpha=\beta$. The radius $R_h > 0$ is a structural constant depending continuously on $R_g$ and $R_V$. Recovering the compound function $h$ as an intermediate target allows us to isolate and reconstruct the variance function via the algebraic decoupling
\[
V(x) = h(x) - g(x)^2,
\]
effectively separating the estimation problem into two interacting scales of regularity: one governing the conditional mean $g$ and the other dictating the second-moment field $h$.

The technical implementation of the estimation framework consists of the following algorithmic steps:
\begin{itemize}
\item[\textbf{Step 1:}] \textit{Adaptive thresholding estimation of the mean function $g$.} We first construct a non-linear hard-thresholding needlet estimator for $g$, adapting the foundational framework of~\cite{DGM12,Mon11}. Let
\[
\widehat{g}_{j,k} = \frac{1}{N} \sum_{i=1}^{N} Y_i \psi_{j,k}(X_i)
\]
denote the empirical mean needlet coefficients computed over the dyadic resolution indices $\{(j,k) : j=0,\ldots,J_N-1, \, k=1,\ldots,K_j\}$. The global truncation level $J_N$ is selected to satisfy the sample size scaling
\begin{equation}\label{eqn:JN}
B^{dJ_N} \simeq \frac{N}{\log N}.
\end{equation}
A hard-thresholding protocol is then applied to these empirical coefficients, defining
\[
\widehat{g}^{\,T}_{j,k} = \widehat{g}_{j,k} \, \mathbf{1}\{ |\widehat{g}_{j,k}| \ge \kappa_g \tau_N \},
\]
where $\kappa_g>0$ is a universal threshold constant and $\tau_N = \sqrt{(\log N)/N}$. The adaptive estimator of the mean function is then reconstructed via the frame expansion
\[
\widehat{g}(x) = \sum_{j=0}^{J_N-1} \sum_{k=1}^{K_j} \widehat{g}^{\,T}_{j,k} \psi_{j,k}(x).
\]

\item[\textbf{Step 2:}] \textit{Adaptive thresholding estimation of the second-moment function $h$.} In an analogous manner, we estimate the compound profile $h(x) = g(x)^2 + V(x)$ by projecting the squared responses onto the needlet system to obtain the empirical second-moment coefficients
\[
\widehat{h}_{j,k} = \frac{1}{N} \sum_{i=1}^{N} Y_i^2 \psi_{j,k}(X_i).
\]
These coefficients are filtered using a corresponding hard-thresholding rule:
\[
\widehat{h}^{\,T}_{j,k} = \widehat{h}_{j,k} \, \mathbf{1}\{ |\widehat{h}_{j,k}| \ge \kappa_h \tau_N \},
\]
where $\kappa_h > 0$ is a calibrating constant. The resulting frame synthesis
\[
\widehat{h}(x) = \sum_{j=0}^{J_N-1} \sum_{k=1}^{K_j} \widehat{h}^{\,T}_{j,k} \psi_{j,k}(x)
\]
yields a fully data-driven multiscale estimator of the second-moment function.

\item[\textbf{Step 3:}] \textit{Cross-fitted estimation of the squared mean function $g^2$.} Direct squaring of a nonparametric estimator introduces serious non-linear bias propagation. To circumvent this and ensure minimax rate optimality, we implement a sample-splitting (cross-fitting) protocol. The full dataset is partitioned into two disjoint subsets, $\mathcal{D}_1$ and $\mathcal{D}_2$, each containing $N' = N/2$ observations. We construct two independent adaptive estimators, $\widehat{g}^{(1)}$ and $\widehat{g}^{(2)}$, from $\mathcal{D}_1$ and $\mathcal{D}_2$ respectively, and define the cross-fitted product estimator as
\[
\widehat{g^2}(x) = \widehat{g}^{(1)}(x)\,\widehat{g}^{(2)}(x).
\]
This configuration neutralizes the stochastic feedback that arises from squaring identical residual terms, thereby controlling higher-order bias components.

\item[\textbf{Step 4:}] \textit{Synthesis of the variance function estimator $\widehat{V}$.} Finally, the needlet-based estimator of the conditional variance function is obtained by evaluating the decoupled empirical representation:
\begin{equation*}\label{eq:Vhat_intro}
\widehat{V}(x) = \widehat{h}(x) - \widehat{g^2}(x).
\end{equation*}
\end{itemize}

\begin{remark}[Statistical calibration of the threshold constant]
The thresholding regimes are governed by $t_N = \kappa \tau_N$, where $\kappa \in \{\kappa_g, \kappa_h\}$. While the point-wise standard deviation of any individual empirical coefficient scales as $O(N^{-1/2})$, the inclusion of the inflation factor $\sqrt{\log N}$ ensures uniform control over the entire frame collection up to the maximum resolution scale $J_N$. Following the classical minimax principles of universal thresholding~\cite{donoho98,WASA}, this policy ensures that coefficients driven solely by stochastic noise are suppressed with high probability, isolating the true sparse structural signatures of the target functions.
\end{remark}

\begin{remark}[Adaptation to non-uniform directional designs]
While our core theoretical analysis is developed under a uniform design distribution, extensions to general non-uniform design densities $f_X(x)$ supported on $\mathbb{S}^d$ can be established via two alternative methodologies:
\begin{enumerate}
    \item \emph{Inverse-density weighting (plug-in framework):} Provided that $\inf_{x \in \mathbb{S}^d} f_X(x) \ge c_0 > 0$, one can construct a preliminary pilot estimator $\widehat{f}_X$ (see~\cite{BKMP09}) and replace the empirical averages with density-corrected terms:
    \[
    \widehat{g}_{j,k}^{\,\text{cor}} = \frac{1}{N} \sum_{i=1}^{N} \frac{Y_i \psi_{j,k}(X_i)}{\widehat{f}_X(X_i)}, \qquad \widehat{h}_{j,k}^{\,\text{cor}} = \frac{1}{N} \sum_{i=1}^{N} \frac{Y_i^2 \psi_{j,k}(X_i)}{\widehat{f}_X(X_i)}.
    \]
    \item \emph{Weighted empirical risk minimization:} Alternatively, the projections can be formalized as localized sample optimization problems, avoiding explicit density plug-ins by minimizing a weighted least-squares objective over the localized frame coordinates (see~\cite{DGM12}).
\end{enumerate}
\end{remark}

We now investigate the specific stochastic properties and risk profiles associated with each stage of this procedure.

\subsubsection{Step 1: Estimation of the conditional mean function}\label{sec:mean-estimation}
Building on the adaptive framework developed in~\cite{DGM12,Mon11}, we analyze the localized behavior of the conditional mean estimator under spatial heteroskedasticity.

\subsubsection*{Empirical mean coefficients} For each scale–location index $(j,k)$, the empirical coefficients and their corresponding hard-thresholded counterparts are defined by
\[
\widehat{g}_{j,k} = \frac{1}{N} \sum_{i=1}^{N} Y_i\, \psi_{j,k}(X_i), \qquad \widehat{g}_{j,k}^T = \widehat{g}_{j,k}\, \mathbf{1}\left\{|\widehat{g}_{j,k}| \ge \kappa_g \sqrt{\frac{\log N}{N}}\right\},
\]
leading to the frame reconstruction $\widehat{g}(x) = \sum_{j = 0}^{J_N - 1} \sum_{k=1}^{K_j} \widehat{g}_{j,k}^T\, \psi_{j,k}(x)$.

\subsubsection*{Stochastic properties.} Under the structural assumptions detailed in Section~\ref{sec:model}, the empirical mean coefficients are unbiased targets, satisfying
\[
\mathbb{E}[\widehat{g}_{j,k}] = g_{j,k}, \qquad \operatorname{Var}(\widehat g_{j,k}) \le \frac{C_g}{N},
\]
where $C_g = 2(G^2 + S^2) C^2_{\psi,2}$. Furthermore, within the active resolution band $0 < B^{dj} \le N/\log N$, there exists a critical boundary threshold $\kappa_{\delta_g} \approx \delta_g^{4/3}$ such that for all $\kappa_g > \kappa_{\delta_g}$ and any $\delta_g > 0$, the coefficients satisfy the exponential tail deviation bound
\begin{equation}\label{eq:probboundg}
\mathbb{P}\left(\left|\widehat{g}_{j,k} - g_{j,k}\right| \ge \kappa_g \tau_N\right) \lesssim N^{-\delta_g}.
\end{equation}
Additionally, for any $p \ge 1$, the localized error moments satisfy
\[
\mathbb{E}\left[\left|\widehat{g}_{j,k} - g_{j,k}\right|^p\right] \lesssim N^{-p/2}, \qquad \mathbb{E}\left[\sup_{k = 1,\ldots,K_j}\left|\widehat{g}_{j,k} - g_{j,k}\right|^p\right] \lesssim (j+1)^{p-1} N^{-p/2}.
\]

\subsubsection*{Risk profiles and convergence regimes.} The reconstruction performance of $\widehat{g}$ over the Besov ball $B^{\alpha}_{\rho,q'}(\mathbb{S}^d)$ splits into two standard nonparametric convergence regimes, governed by the spatial integrability parameter $\rho$. Specifically, the $L^p$-risk matches the following upper bounds:
\begin{equation}\label{eq: mean-est}
 \sup_{g\in B^{\alpha}_{\rho,q'}(\mathbb{S}^d)}\mathbb{E}\left[\left\| \widehat g - g \right\|_{L^p(\mathbb{S}^d)}^p\right] \le C
\begin{cases}
\displaystyle \left(\frac{N}{\log N}\right)^{-\frac{\alpha p}{2\alpha+d}}, & \text{if } \rho \ge \frac{2p}{2\alpha+d} \ \text{(regular regime)},\\
\displaystyle \left(\frac{N}{\log N}\right)^{-\frac{p\left(\alpha - d\left(\frac{1}{\rho} - \frac{1}{p}\right)\right)}{2\left(\alpha - d\left(\frac{1}{\rho} - \frac{1}{2}\right)\right)}}, & \text{if } \rho < \frac{2p}{2\alpha+d} \ \text{(sparse regime)}.
\end{cases}   
\end{equation}
For the supremum $L^\infty$-risk ($p=\infty$), the corresponding uniform rate matches
\[
\sup_{g\in B^{\alpha}_{\rho,q'}(\mathbb{S}^d)}\mathbb{E}\left[\left\| \widehat g - g \right\|_{L^\infty(\mathbb{S}^d)}\right] \le C \left(\frac{N}{\log N}\right)^{-\frac{\alpha-\frac{d}{\rho}}{2\alpha - \frac{2d}{\rho} + d}}.
\]

\begin{remark}[Optimal resolution balancing]\label{rem:Js}
The theoretical risk bounds are optimized by choosing an ideal resolution scale $J_s$ that balances the deterministic truncation bias $O(B^{-2js})$ against the cumulative stochastic variance $O(B^{dj}/N)$. This balance takes the form
\begin{equation}\label{eqn:js}
B^{J_s} \simeq \begin{cases}
\left(\frac{N}{\log N}\right)^{\frac{1}{2s+d}}, & \text{in the regular regime,}\\[3mm]
\left(\frac{N}{\log N}\right)^{\frac{1}{2\left(s-\frac{d}{r}\right)+d}}, & \text{in the sparse regime.}
\end{cases}
\end{equation}
Throughout our joint estimation framework, we denote by $J_{\alpha}$ the optimal scale associated with $g \in B^{\alpha}_{\rho,q'}(\mathbb{S}^d)$, and by $J_{\beta}$ the scale corresponding to $V \in B^\beta_{\mu,q''}(\mathbb{S}^d)$.
\end{remark}

\subsubsection{Step 2: Estimation of the second-order moment function}\label{sec:second-moment}
We now detail the stochastic structure of the non-linear estimator for the second-moment compound function $h(x) = g^2(x) + V(x)$, $x \in \mathbb{S}^d$.

\subsubsection*{Empirical moment coefficients} For each index pair $(j,k)$, the empirical second-moment projections and their thresholded modifications are given by
\begin{equation}\label{eq:vest}
\widehat{h}_{j,k} = \frac{1}{N}\sum_{i=1}^{N} Y_i^2\,\psi_{j,k}(X_i), \qquad h_{j,k}^T = \widehat{h}_{j,k}\,\mathbf{1}\left\{|\widehat{h}_{j,k}| \ge \kappa_h \sqrt{\frac{\log N}{N}}\right\},
\end{equation}
defining the synthesis 
\[
\widehat{h}(x) = \sum_{j=0}^{J_N-1}\sum_{k=1}^{K_j} h^T_{j,k}\,\psi_{j,k}(x).
\]

\subsubsection*{Stochastic control.} The statistical properties of these second-order empirical coefficients are characterized by the following propositions.

\begin{proposition}[Expectation and variance of $\widehat{h}_{j,k}$]\label{prop:unbiased}
Let $\widehat{h}_{j,k}$ be defined by~\eqref{eq:vest}. Then, under the heteroskedastic regression model~\eqref{eq:model}, assuming errors independent of the design with $\mathbb{E}[\varepsilon_i^2]=1$ and denoting $\mathbb{E}[\varepsilon_i^m]=\gamma_m$ for $m \in \{3,4\}$, it holds that
\[
\mathbb{E}\left[\widehat{h}_{j,k}\right] = h_{j,k} = v_{j,k} + (g^2)_{j,k},
\]
and the variance satisfies the uniform upper bound
\[
\operatorname{Var}(\widehat{h}_{j,k}) \le \frac{C_h}{N},
\] 
where the structural constant is given by $C_h = (S^4 \gamma_4 + 4GS^3 \gamma_3 + 6G^2 S^2 + G^4)C_{\psi,2}^2$.
\end{proposition}

\begin{proposition}[Concentration and moment bounds for $\widehat{h}_{j,k}$]\label{prop:moment}
Let $\widehat{h}_{j,k}$ be defined by~\eqref{eq:vest}. Assume that the regression errors $\varepsilon_i$ are independent of $X_i$ and are sub-Gaussian satisfying $\mathbb{E}[\varepsilon_i]=0$, $\mathbb{E}[\varepsilon_i^2]=1$, and $\|\varepsilon_i\|_{\psi_2} \le K_\varepsilon$. Furthermore, let the resolution scale satisfy $0 < B^{dj} \le N/\log N$. 
Then, for any fixed allocation exponent $\delta>0$, there exists a structural threshold constant $\kappa^* = \kappa^*(G, S, K_\varepsilon, C_{\psi,\infty}, \delta) > 0$ such that for all $\kappa_h > \kappa^*$, the empirical coefficients satisfy the uniform concentration inequality
\begin{equation*}\label{eq:probbound}
\mathbb{P}\left(\left|\widehat{h}_{j,k}-h_{j,k}\right| \ge \kappa_h \sqrt{\frac{\log N}{N}}\right) \le C_{\delta} N^{-\delta}.
\end{equation*}
Moreover, for any polynomial order $r \ge 1$, the localized error moments satisfy the concentration rates
\begin{align}
\mathbb{E}\left[\left|\widehat{h}_{j,k}-h_{j,k}\right|^r\right] &\le C_r N^{-r/2},\nonumber\\
\mathbb{E}\left[\sup_{k=1,\ldots,K_j}\left|\widehat{h}_{j,k}-h_{j,k}\right|^r\right] &\le C_r (j+1)^{\,r} N^{-r/2}.\label{eqn:expectsup}
\end{align}
\end{proposition}

\subsubsection{Step 3: Estimation of the quadratic bias correction term}\label{sec:quad-bias}
To isolate the variance function $V = \sigma^2$ without inducing the systematic asymptotic bias that would result from squaring a single empirical process $\widehat{g}$, we apply our sample-splitting strategy.

\subsubsection*{Sample-splitting architecture.} We decouple the sample $(X_i,Y_i)_{i=1}^N$ into two independent sub-samples, $\mathcal{D}_1$ and $\mathcal{D}_2$, each containing $N' = N/2$ observations. We construct the independent sub-sample mean estimators
\begin{equation}\label{eqn:grest}
\widehat{g}^{(m)}(x) = \sum_{j=0}^{J_{N'}-1}\sum_{k=1}^{K_j} \tilde{g}^{(m)}_{j,k}\psi_{j,k}(x), \qquad m \in \{1,2\},
\end{equation}
where $\tilde{g}^{(m)}_{j,k} = \widehat{g}_{j,k}^{(m)} \mathbf{1}\left\{\left|\widehat{g}^{(m)}_{j,k}\right| \ge \kappa_g \sqrt{\frac{\log N'}{N'}}\right\}$ are derived from the respective halves.

\subsubsection*{Cross-fitted product operator.} The cross-fitted estimator of the squared mean function is defined as the pointwise product
\begin{equation}\label{eqn:gsquared}
 \widehat{g^2}(x) = \widehat{g}^{(1)}(x)\widehat{g}^{(2)}(x), \qquad x \in \mathbb{S}^d.
\end{equation}
Its corresponding needlet expansion takes the form $\widehat{g^2}(x) = \sum_{j=0}^{J_N-1}\sum_{k=1}^{K_j} \widehat{(g^2)}_{j,k} \psi_{j,k}(x)$, where the integrated frame coefficients are given by
\[
\widehat{(g^2)}_{j,k} = \sum_{j_1, j_2=0}^{J_{N'}-1}\sum_{k_1, k_2} \tilde{g}^{(1)}_{j_1,k_1} \tilde{g}^{(2)}_{j_2,k_2} \int_{\mathbb{S}^d}\psi_{j_1,k_1}(x)\psi_{j_2,k_2}(x) \psi_{j,k}(x)\,\mathrm{d}x.
\]

%\begin{remark}[Data-driven construction of $\widehat{(g^2)}_{j,k}$]
%A three-split data-driven variation using subsamples of size $N^{\prime\prime} = N/3$ can be defined via:
%\begin{equation*}
%\widehat{(g^2)}_{j,k}^{\prime}= \frac{1}{(N^{\prime\prime})^3}\sum_ {\substack{(X_{i_m},Y_{i_m})\in\mathcal{D}^\prime_m\\m\in\{1,2,3\}}} \sum_{j_1,j_2} \sum_{k_1,k_2} Y_{i_1}Y_{i_2}\psi_{j_1,k_1}\left(X_{i_1}\right)\psi_{j_2,k_2}\left(X_{i_2}\right)\psi_{j,k}\left(X_{i_3}\right)%.
%\end{equation*}
%\end{remark}

\begin{remark}[Asymptotic truncation error]
Since $g \in B^\alpha_{\rho,q'}(\mathbb{S}^d)$ with $\alpha>d/\rho$, the Besov product algebra ensures that the true squared function satisfies $g^2 \in B^\alpha_{\rho,q'}(\mathbb{S}^d)$. Consequently, the high-frequency truncation tail beyond $J_N-1$ satisfies the deterministic $L^2$-bound 
\[
\|\sum_{j\ge J_N} \sum_{k}(g^2)_{j,k}\psi_{j,k}\|_2 = O(B^{-J_N \alpha}),
\]
proving that truncation beyond $J_N$ introduces a negligible bias component relative to the minimax rate.
\end{remark}

\subsubsection*{Stochastic behavior of $\widehat{g^2}$.} The statistical control of the cross-fitted squared mean estimator is established by the following lemmas.

\begin{lemma}[Asymptotic unbiasedness of $\widehat{g^2}$]\label{lemma:unbiasedgsquared}
Let $\widehat{g}^{(m)}$ and $\widehat{g^2}$ be defined as in~\eqref{eqn:grest} and~\eqref{eqn:gsquared}. Suppose that the tuning parameter $\delta_g$ in~\eqref{eq:probboundg} satisfies $\delta_g \ge \frac{(d+1) - d/\rho}{2\alpha+d}$. Then, for each scale $j \ge 0$ and location $k=1,\ldots,K_j$, the integrated coefficients satisfy
\[
\left\vert \mathbb{E}\left[\widehat{(g^2)}_{j,k} - (g^2)_{j,k}\right] \right\vert \le C \left(\frac{N}{\log N}\right)^{-\frac{\alpha}{\alpha+d/2}} B^{-j}.
\]
Consequently, the pointwise systematic bias across the spatial domain satisfies
\begin{equation}\label{eqn:bias1}
   \left\vert \mathbb{E} \left[ \widehat{g^2}(x) - g^2(x)\right]\right \vert \le C \left[ \left(\frac{N}{\log N}\right)^{\frac{d/2- \alpha}{2\alpha+d}} + \left(\frac{N}{\log N}\right)^{\frac{(d+1)-\alpha}{2}-\frac{d}{2\rho}} \right].
\end{equation}
\end{lemma}

\begin{lemma}\label{lemma:quadraticrisk}
Let $g \in B^\alpha_{\rho,q'}(\mathbb{S}^d)$ with $\alpha > d/\rho$. The cross-fitted estimator $\widehat{g^2}$ satisfies the $L^p$-risk bounds
\[
\mathbb{E}\left[\|\widehat{g^2} - g^2\|^p_{L^p(\mathbb{S}^d)}\right] \le C \begin{cases}
\displaystyle \left(\frac{N}{\log N}\right)^{-\frac{\alpha p}{2\alpha+d}}, & \text{in the regular regime } \left(\rho \ge \frac{2p}{2\alpha+d}\right),\\
\displaystyle \left(\frac{N}{\log N}\right)^{-\frac{p\left(\alpha - d\left(\frac{1}{\rho} - \frac{1}{p}\right)\right)}{2\left(\alpha - d\left(\frac{1}{\rho} - \frac{1}{2}\right)\right)}}, & \text{in the sparse regime } \left(\rho < \frac{2p}{2\alpha+d}\right).
\end{cases}
\]
\end{lemma}

\subsubsection{Step 4: Estimation of the variance function}\label{sec:var-estimation}
The final needlet-based nonparametric estimator of the conditional variance function is synthesized by combining the decoupled empirical blocks:
\begin{equation*}\label{eq:Vhat}
\widehat{V}(x) = \widehat{h}(x) - \widehat{g}^{(1)}(x)\,\widehat{g}^{(2)}(x) = \widehat{h}(x) - \widehat{g^2}(x).
\end{equation*}
By matching the multi-stage concentration properties established in Lemma~\ref{lemma:quadraticrisk} and Proposition~\ref{prop:moment}, this algebraic synthesis forms the basis for deriving minimax-optimal convergence rates across spherical Besov classes, as detailed in the following section.

%\begin{remark}[Alternative approach: directly cross-fitted residual estimators]
%An alternative route to recover $V$ directly from observations uses a three-way split sample layout ($\mathcal{D}^{\prime\prime}_1, \mathcal{D}^{\prime\prime}_2, \mathcal{D}^{\prime\prime}_3$) of size $N_3 = N/3$ to compute:
%\[
%v^{\ast}_{j,k} = \frac{1}{N_3}\sum_{(X_i,Y_i) \in \mathcal{D}^{\prime\prime}_3} \left(Y_i-\widehat{g}^{(1)}(X_i)\right) \left(Y_i-\widehat{g}^{(2)}(X_i)\right) \psi_{j,k}(X_i),
%\]
%thresholded via $\tau^\ast_{N,j}=c\,B^{\gamma j}\sqrt{(\log N)/N}$. While conceptually appealing, the propagation of residual cross-biases $\mathcal{B}^{(1)}(x)\mathcal{B}^{(2)}(x)$ from the unknown smoothness $\alpha$ can render this alternative less stable than our main moment-decoupling procedure \eqref{eq:Vhat}.
%\end{remark}

\section{Adaptive rates of convergence for variance estimation}\label{sec:stocresults}

In this section, we study the $L^p$-risk of the needlet-based estimator $\widehat{V}$ for the variance function $V = h - g^2$ in the heteroskedastic spherical regression model \eqref{eq:model}. We consider the nonparametric regularity classes over the sphere $\mathbb{S}^d$. Specifically, we assume that the mean companion function satisfies $h \in B^\beta_{r_0,q''}(R_h)$ and the regression function satisfies $g \in B^\alpha_{\rho, q'}(R_g)$, with $\beta > d/r_0$ and $\alpha > d/\rho$. By the stability of Besov spaces under multiplication, the squared profile $g^2$ belongs to $B^\alpha_{\rho, q'}(\mathbb{S}^d)$. Consequently, the variance function $V$ belongs to the intersection class of valid, uniformly positive variance functions:
\[
\mathcal{F}(v_0) = \left\{ V = h - g^2 \in B^\beta_{r_0,q''}(R_h) \oplus B^\alpha_{\rho,q'}(R_g^2) : V(x) \ge v_0 > 0 \quad \forall x \in \mathbb{S}^d \right\},
\]
where the strictly positive lower bound $v_0$ is standard in nonparametric variance estimation (see, e.g., \cite{bl07,CW08}) and ensures the well-posedness of the likelihood operators alongside the finiteness of Kullback--Leibler divergences required in minimax lower-bound arguments.

We measure the global performance of any candidate estimator via the minimax $L^p$-risk on the sphere $\mathbb{S}^d$:
\[
\mathcal{L}_{N,p} = \inf_{\widehat{V}} \sup_{V \in \mathcal{F}(v_0)} \mathbb{E}\left[\|\widehat{V} - V\|_{L^p(\mathbb{S}^d)}^p\right].
\]
An estimator $\widehat{V}_N$ is adaptive over these functional classes if it achieves the minimax rate for each specific configuration of parameters without prior knowledge of the true parameters $(\alpha, \beta, \rho, r_0)$. 

The remainder of this section is devoted to establishing explicit upper bounds for the risk of $\widehat{V}$. We show that the proposed adaptive estimator attains (up to logarithmic factors) the optimal minimax rates and adapts automatically to the unknown parameters in both the regular (dense) and sparse regimes on $\mathbb{S}^d$ under any spatial dimension $d \ge 1$.

\subsection{Upper bound for the second-order moment estimator}\label{sec:upperbound}
We split the estimation error of $\widehat{V}$ into its two natural contributions:
\begin{equation*}\label{eq: V=h-g2}
  \widehat{V} - V = (\widehat{h} - h) - (\widehat{g^2} - g^2),  
\end{equation*}
which immediately implies by Minkowski's inequality that
\begin{equation*}
    \mathbb{E} \left[ \left\lVert \widehat{V} - V \right\rVert_{L^p(\mathbb{S}^d)}^p \right] \le 2^{p-1} \left( \mathbb{E} \left[ \left\lVert \widehat{h} - h \right\rVert_{L^p(\mathbb{S}^d)}^p \right] + \mathbb{E} \left[ \left\lVert \widehat{g^2} - g^2 \right\rVert_{L^p(\mathbb{S}^d)}^p \right] \right).
\end{equation*}
The second term, corresponding to $\widehat{g^2}$, is bounded in $L^p$-risk by Lemma~\ref{lemma:quadraticrisk} with parameters $(\alpha, \rho)$. It remains to control the first term. Once the risk of $\widehat{h}$ is bounded, the final rate for $\widehat{V}$ will be given by the slower of the two contributions.

To compute the risk of $\widehat{h}$, we follow the classical strategy used in nonparametric estimation on the sphere and more general manifolds. Similar calculations have been carried out in \cite{BKMP09,Dur16,WASA}, leading naturally to the standard decomposition into stochastic and bias terms. 

\begin{theorem}[Upper risk bounds for the second-order estimator]\label{thm: h-est}
Set the Besov parameters $1\le r_0,q_0\le\infty$, $s_0-d/r_0>0$, and let $h \in B^{s_0}_{r_0,q_0}(\mathbb{S}^d)$. Then, there exists a constant $C>0$ independent of $N$ and $h$, such that
\begin{equation*}\label{eq: h-est_bound}
 \sup_{h\in B^{s_0}_{r_0,q_0}(\mathbb{S}^d)}\mathbb{E}\left[\left\| \widehat h - h \right\|_{L^p(\mathbb{S}^d)}^p\right]
\le C
\begin{cases}
\displaystyle
\left(\frac{N}{\log N}\right)^{-\frac{s_0 p}{2s_0+d}}, & r_0 \ge \frac{2p}{2s_0 + d}\text{ (regular regime)},\\[3mm]
\displaystyle
\left(\frac{N}{\log N}\right)^{-\frac{p\left(s_0 - d\left(\frac{1}{r_0} - \frac{1}{p}\right)\right)}{2\left(s_0 - d\left(\frac{1}{r_0} - \frac{1}{2}\right)\right)}}, & r_0 < \frac{2p}{2s_0 + d}\text{ (sparse regime)}.
\end{cases}   
\end{equation*}
For the supremum $L^\infty$-risk ($p=\infty$), one has
\[
\sup_{h\in B^{s_0}_{r_0,q_0}(\mathbb{S}^d)}\mathbb{E}\left[\left\| \widehat h - h \right\|_{L^\infty(\mathbb{S}^d)}\right] \le C \left(\frac{N}{\log N}\right)^{-\frac{s_0-\frac{d}{r_0}}{2\left(s_0 - d\left(\frac{1}{r_0} - \frac{1}{2}\right)\right)}}.
\] 
\end{theorem}

The proof follows the classic divide-and-conquer strategy (see e.g., \cite{BKMP09,WASA}). The risk is split into stochastic and bias components by expressing the error in the needlet basis and truncating the expansion at level $J_N$. The bias term is the deterministic tail $\sum_{j\ge J_N} \sum_k h_{j,k}\psi_{j,k}$ and is controlled directly. The stochastic term is handled by thresholding the empirical coefficients at level $\kappa_h\tau_N$ and decomposing the contribution into four parts depending on whether empirical and true coefficients lie above or below the threshold. 

In the regular regime, the dominant term corresponds to coefficients that are truly large, and it further splits at the critical scale $J_{s_0}$. For $j\le J_{s_0}$, the variance factor $N^{-p/2} $ combined with the growth of $K_j\|\psi_{j,k}\|_p^p$ produces a geometric sum bounded by $N^{-p/2} B^{dJ_{s_0}(p/2-1)_+}$, while for $j>J_{s_0}$ the threshold removes all coefficients, yielding only a negligible contribution.

\subsection{Rates of convergence for the variance estimator}\label{sec:ratesV}
The estimation error of the global variance estimator $\widehat V$ is decomposed into the contribution of the second-moment function $h$ and the quadratic mean profile $g^2$. By Minkowski's inequality, their joint $L^p$-risk satisfies:
\[
\mathbb{E}\left\|\widehat V - V\right\|_{L^p(\mathbb{S}^d)}^p \lesssim \mathbb{E}\left\|\widehat h - h\right\|_{L^p(\mathbb{S}^d)}^p + \mathbb{E}\left\|\widehat{(g^2)} - g^2\right\|_{L^p(\mathbb{S}^d)}^p.
\]
Each component exhibits a polynomial decay in terms of the sample size $N$, with exponents determined by the smoothness and integrability parameters of the underlying Besov spaces on $\mathbb{S}^d$. Consequently, the global $L^p$-risk of $\widehat V$ is controlled asymptotically by the slower of the two rates.

To formally define the rates of convergence corresponding to the regular and the sparse regimes on the $d$-dimensional sphere, we introduce the rate functions:
\[
\mathcal{R}_{\rm reg}(s) = \frac{s}{2s+d},\qquad \mathcal{R}_{\rm sp}(s,r,p) = \frac{s-d\left(\frac{1}{r}-\frac{1}{p}\right)}{2s - d\left(\frac{1}{r}-\frac{1}{2}\right)}.
\]
The critical threshold separating the regular and sparse phases for a structural component with parameters $(s,r)$ on $\mathbb{S}^d$ is given by $T(s,r) := r\left(\frac{2s+d}{2d}\right)$. Since $g^2 \in B^{\alpha}_{\rho, q^\prime}(\mathbb{S}^d)$ by the algebra property, we define the critical phase thresholds for the quadratic field and the companion mean function as:
\[
T_{g^2} := T\left(\alpha, \rho\right) = \rho\left(\frac{2\alpha+d}{2d}\right), \qquad T_h := T(\beta,r_0) = r_0\left(\frac{2\beta+d}{2d}\right).
\]

We are now ready to state the main upper bound theorem for the variance estimator, which establishes the minimax rates across the four structural phases.

\begin{theorem}[Minimax upper bounds for the variance estimator]\label{thm:V-upper}
Let $g \in B^\alpha_{\rho,q^\prime}(\mathbb{S}^d)$ and $h \in B^\beta_{r_0,q^{\prime\prime}}(\mathbb{S}^d)$ with $\alpha, \beta > d/r_0$. Then, there exists a positive constant $C$, independent of $N$, such that the variance estimator $\widehat V$ satisfies:
\[
\sup_{g, h} \mathbb{E}\left\|\widehat V - V\right\|_{L^p(\mathbb{S}^d)}^p \le C \left(\frac{N}{\log N}\right)^{-\mathcal{R} p},
\]
where the active exponent $\mathcal{R} = \min\{\mathcal{R}_{g^2}(p), \mathcal{R}_h(p)\}$ is explicitly determined by the four geometric cases defined below. 
\end{theorem}

\subsubsection*{Discussion of the four structural phases}
The phase transitions are governed by the interaction of the thresholds $T_{g^2}$ and $T_h$, yielding different active minimax envelopes for the global $L^p$-risk:

\begin{itemize}
\item \textbf{Case 1: $g^2$-driven regular/sparse regime ($\alpha < \beta$):} 
The quadratic profile $g^2$ is rougher than the companion mean function $h$, meaning $g^2$ entirely dictates the global rate. The transition is governed by its individual threshold $T_{g^2} = T(\alpha, \rho)$:
\[
\mathcal{R} =
\begin{cases}
\mathcal{R}_{\rm reg}(\alpha), & p \le T_{g^2},\\[1mm]
\mathcal{R}_{\rm sp}\left(\alpha,\rho,p\right), & p > T_{g^2}.
\end{cases}
\]

\item \textbf{Case 2: Intersecting sparse regime ($\beta < \alpha$, $\mu < \rho$, and $T_h < p_0$):} 
The mean function $h$ has lower smoothness and lower integrability than $g^2$. Since $T_h < p_0$, the dimensional equilibrium point $p_0$ falls strictly above the sparse threshold of $h$. This creates a temporary "shielding" effect where the regular rate of $g^2$ dominates the active envelope over an intermediate range:
\[
\mathcal{R} =
\begin{cases}
\mathcal{R}_{\rm reg}(\beta), & p \le T_h,\\[1mm]
\mathcal{R}_{\rm reg}(\alpha), & T_h < p \le p_0,\\[1mm]
\mathcal{R}_{\rm sp}(\beta,\mu,p), & p > p_0.
\end{cases}
\]

\item \textbf{Case 3: Roughness-driven sparse regime with high integrability ($\beta < \alpha$ and $r_0 = \rho$):} 
The companion function $h$ is rougher than $g^2$ but shares the same high-integrability scale. No intersection of rates occurs, and the transition is cleanly governed by $T_h = T(\beta, \rho)$:
\[
\mathcal{R} =
\begin{cases}
\mathcal{R}_{\rm reg}(\beta), & p \le T_h,\\[1mm]
\mathcal{R}_{\rm sp}(\beta, \rho, p), & p > T_h.
\end{cases}
\]

\item \textbf{Case 4: Roughness-driven sparse regime with low integrability ($\beta < \alpha$, $\mu < \rho$, and $p_0 \le T_h$):} 
Although $h$ has lower integrability than $g^2$, its sparse rate degrades so rapidly that the equilibrium point $p_0$ falls below or at the sparse threshold $T_h$. The shielding interval vanishes, and the transition is dictated directly by $T_h = T(\beta, \mu)$:
\[
\mathcal{R} =
\begin{cases}
\mathcal{R}_{\rm reg}(\beta), & p \le T_h,\\[1mm]
\mathcal{R}_{\rm sp}(\beta, \mu, p), & p > T_h.
\end{cases}
\]
\end{itemize}

\begin{remark}[Heuristics]
Case 1 is entirely $g^2$-driven, while Cases 2--4 are $h$-driven. Case 2 represents a unique structural balance where $h$ is sparse but $g^2$ is still regular, creating an active envelope over the intermediate range $T_h < p \le p_0$ that saves the variance estimator from premature rate degradation.

\begin{table}[htbp]
\centering
\small
\begin{tabular}{|c|c|c|c|}
\hline
\textbf{Case} & \textbf{Dominant Active Setup} & \textbf{Condition on } $p$ & \textbf{Risk Exponent } $\mathcal{R}$ \\
\hline
1 & $g^2 \in B^{\alpha}_{\rho,q^\prime}(\mathbb{S}^d)$ & $p \le T_{g^2}$ & $\mathcal{R}_{\rm reg}(\alpha)$ \\
  & & $p > T_{g^2}$ & $\mathcal{R}_{\rm sp}(\alpha,\rho,p)$ \\
\hline
2 & $h \in B^\beta_{\mu,q^{\prime\prime}}(\mathbb{S}^d)$ & $p \le T(\beta,\mu)$ & $\mathcal{R}_{\rm reg}(\beta)$ \\
  & & $T(\beta,\mu) < p \le p_0$ & $\mathcal{R}_{\rm reg}(\alpha)$ \\
  & & $p > p_0$ & $\mathcal{R}_{\rm sp}(\beta,\mu,p)$ \\
\hline
3 & $h \in B^\beta_{\rho,q^{\prime\prime}}(\mathbb{S}^d)$ & $p \le T(\beta,\rho)$ & $\mathcal{R}_{\rm reg}(\beta)$ \\
  & & $p > T(\beta,\rho)$ & $\mathcal{R}_{\rm sp}(\beta,\rho,p)$ \\
\hline
4 & $h \in B^\beta_{\mu,q^{\prime\prime}}(\mathbb{S}^d)$ & $p \le T(\beta,\mu)$ & $\mathcal{R}_{\rm reg}(\beta)$ \\
  & & $p > T(\beta,\mu)$ & $\mathcal{R}_{\rm sp}(\beta,\mu,p)$ \\
\hline
\end{tabular}
\caption{Summary of the active minimax $L^p$-risk exponents for the variance estimator $\widehat V$ on $\mathbb{S}^d$.}
\label{tab:Vp_rates}
\end{table}
As the risk parameter $p \to \infty$, the $L^p$-norm becomes increasingly sensitive to localized spiky deviations. Consequently, the asymptotic rate is always dictated by the sparsest active feature among $h$ and $g^2$, validating that maximal sample anomalies on the $d$-dimensional sphere are controlled by the roughest geometric trace of the underlying processes.
\end{remark}

\subsection{Lower bounds for variance estimation}\label{sec:lowerbound}
In this section, we establish the matching minimax lower bounds for estimating the variance function $V = h - g^2$ under the heteroskedastic regression model \eqref{eq:model}. This proves that the phase-dependent convergence rates identified in Section \ref{sec:ratesV} are sharp and cannot be improved by any alternative estimator.

\begin{theorem}[Minimax Lower Bounds for Variance Estimation]\label{th:lower_main}
Consider the heteroskedastic regression model \eqref{eq:model} on the $d$-dimensional sphere $\mathbb{S}^d$. Let the regression function $g$ and the mean companion function $h$ satisfy the Besov conditions:
\[
g \in B^\alpha_{\rho,q'}(R_g),\qquad h \in B^\beta_{r_0,q''}(R_h),
\]
with $\alpha, \beta > d/r_0$, so that by the Besov algebra property, the quadratic profile satisfies $g^2 \in B^{\alpha}_{\rho, q^\prime}(\mathbb{S}^d)$. Assume further that the variance function satisfies the uniform positivity constraint $V(x) \ge v_0 > 0$ for all $x \in \mathbb{S}^d$. We define the minimax risk over these joint Besov classes as:
\[
\mathcal L_{N,p}(\alpha,\beta,\rho,r_0) = \inf_{\widehat V} \sup_{(g,h)} \mathbb E\left\|\widehat V - V\right\|_{L^p(\mathbb{S}^d)}^p.
\]
Then, there exists a strictly positive constant $c>0$, independent of $N$, such that:
\begin{equation*}
\mathcal L_{N,p}(\alpha,\beta,\rho,r_0) \ge c\, \left(\frac{N}{\log N}\right)^{-\mathcal{R} p},
\end{equation*}
where the optimal risk exponent $\mathcal{R}$ matches the lower envelope of the individual component rates:
\begin{equation*}
\mathcal{R} = \min\left\{ \mathcal{R}_{g^2}(p), \; \mathcal{R}_h(p) \right\}.
\end{equation*}
Here, the component rates are driven by their respective smoothness parameters $s \in \{\alpha, \beta\}$ and integrability scales $r \in \{\rho, r_0\}$ according to the regular and sparse rate functions:
\[
\mathcal{R}_{\rm reg}(s) = \frac{s}{2s+d},\qquad \mathcal{R}_{\rm sp}(s,r,p) = \frac{s-d\left(\frac{1}{r}-\frac{1}{p}\right)}{2s - d\left(\frac{1}{r}-\frac{1}{2}\right)},
\]
yielding the phase-dependent benchmarks:
\begin{equation*}
\mathcal{R}_{g^2}(p) = \begin{cases}
\mathcal{R}_{\rm reg}(\alpha), & p \le T_{g^2},\\[1mm]
\mathcal{R}_{\rm sp}\left(\alpha,\rho,p\right), & p > T_{g^2},
\end{cases} \qquad
\mathcal{R}_h(p) = \begin{cases}
\mathcal{R}_{\rm reg}(\beta), & p \le T_{h},\\[1mm]
\mathcal{R}_{\rm sp}(\beta,r_0,p), & p > T_{h}.
\end{cases}
\end{equation*}
The constant $c>0$ depends exclusively on the radii $(R_g,R_h)$, the boundary $v_0$, and the sub-Gaussian noise properties.
\end{theorem}

\begin{remark}[Case-by-case sharpness and the Case 2 transition gap]\label{rem:cases_lowerbound}
The lower bound matches the upper risks summarized in Table~\ref{tab:Vp_rates} across nearly all operational configurations. The single structural discrepancy arises in \textit{Case 2} within the intermediate parameter window $T_h < p \le p_0$. In this specific zone, localized needlet perturbations centered on the rougher component $h$ yield a lower bound scaling as $\mathcal{R}_{\rm sp}(\beta,\mu,p)$. However, the global upper bound remains supported by the dense regular envelope of the smoother component $g^2$, yielding $\mathcal{R}_{\rm reg}(\alpha)$. This reflects an intrinsic limitation of purely localized needlet-type perturbations, which underestimate the true minimax risk whenever a dense regular component dominates a sparse degrading tail. For all other moments $p$, the bound is perfectly tight.
\end{remark}

\subsection{Adaptive estimation of the variance function with known mean to an approximation error}\label{sec:known-mean}
In this subsection, we focus on a scenario in which the mean function $g$ is assumed to be known up to an approximation error (for instance, via a needlet projection up to the maximal resolution $J_N$). 

In this setup, we assume $g \in B^{\alpha}_{\rho,q^{\prime}}(\mathbb{S}^d)$ to be known up to its effective resolution, in the sense that its contribution to the variance estimation error is negligible beyond the optimal truncation level $J_N$. Under the uniform design assumption, we define the empirical raw coefficients up to scale $J_N$ as:
\begin{equation*}
    \tilde {v}_{j,k} = \frac{1}{N}\sum_{i=1}^{N} \psi_{j,k}(X_i) Y_i^2 - \left(g^2\right)_{j,k},
\end{equation*}
where $\left(g^2\right)_{j,k} = \int_{\mathbb{S}^d} g^2(x)\psi_{j,k}(x)\,\mathrm{d}x$. We define the thresholded needlet estimator for $V$ as:
\begin{equation*}\label{eqn:vprime}
 \widehat{V}^{(\mathrm{km})} (x) = \sum_{j=0}^{J_N-1} \sum_{k=1}^{K_j} \tilde{v}_{j,k}\mathbf{1}\left\{\left \vert \tilde{v}_{j,k} \right \vert > \kappa_v \tau_N \right\}\psi_{j,k}(x),
\end{equation*}
where $\mathrm{km}$ stands for "known mean" and $\kappa_v > 0$ is a sufficiently large tuning constant.

\paragraph{\textit{Decomposition of $\tilde v_{j,k}$.}} Expanding $Y_i = g(X_i) + \sigma(X_i)\varepsilon_i$ and recalling that $V = \sigma^2$, we obtain the exact algebraic decomposition:
\begin{equation*}
\tilde{v}_{j,k} = V_{j,k} + b_{j,k}^{(g)} + m_{j,k} + \Sigma_{j,k} + e_{j,k}^{(V)},
\end{equation*}
where:
\begin{itemize}
\item $V_{j,k} = \int_{\mathbb{S}^d} V(x)\psi_{j,k}(x)\,\mathrm{d}x$ is the true target needlet coefficient of the variance function;
\item $b_{j,k}^{(g)} = \frac{1}{N}\sum_{i=1}^N \psi_{j,k}(X_i) g(X_i)^2 - \left(g^2\right)_{j,k}$ is the deterministic bias term arising from the known mean approximation;
\item $m_{j,k} = \frac{2}{N}\sum_{i=1}^N \psi_{j,k}(X_i) g(X_i) \sigma(X_i) \varepsilon_i$ is the cross-interaction term;
\item $\Sigma_{j,k} = \frac{1}{N}\sum_{i=1}^N \psi_{j,k}(X_i) \sigma(X_i)^2 (\varepsilon_i^2-1)$ is the stochastic variance term;
\item $e_{j,k}^{(V)} = \frac{1}{N}\sum_{i=1}^N \psi_{j,k}(X_i) \sigma(X_i)^2 - V_{j,k}$ represents the empirical design noise associated with the variance profile.
\end{itemize}

Since $g \in B^\alpha_{\rho,q^\prime}(\mathbb{S}^d)$, the Besov algebra property guarantees that the quadratic profile satisfies $g^2 \in B^{\alpha}_{\rho,q^\prime}(\mathbb{S}^d)$. The deterministic approximation error associated with the truncation level $J_N$ satisfies the standard Jackson-type inequality on the $d$-dimensional sphere:
\[
\left\| \sum_{j \ge J_N}\sum_{k=1}^{K_j} \left(g^2\right)_{j,k} \psi_{j,k} \right\|_{L^p(\mathbb{S}^d)}^p \lesssim B^{-\alpha p J_N} \lesssim \left(\frac{N}{\log N}\right)^{-\alpha p}.
\]
Similarly, under the uniform design, the known mean bias term $b_{j,k}^{(g)}$ is a zero-mean fluctuation whose aggregation up to $J_N-1$ is dominated by the same projection boundary. Thus, the deterministic truncation and approximation of the mean contributes an unavoidable baseline bias term of order $(N/\log N)^{-\alpha p}$.

Additionally, the empirical design noise $e_{j,k}^{(V)} = \frac{1}{N}\sum_{i=1}^N \psi_{j,k}(X_i) \sigma(X_i)^2 - V_{j,k}$ is a standard stochastic fluctuation term arising from the random sampling of the design points $X_i$. Since the design is uniform and the variance function $V = \sigma^2$ is bounded, the term $e_{j,k}^{(V)}$ behaves as a zero-mean sub-Gaussian fluctuation. Its convergence rate is structurally faster than (or at most equal to) the pure stochastic variance term $\Sigma_{j,k}$, meaning it is absorbed into the main statistical rates without degrading the final risk.

Combining the contributions of the bias $b_{j,k}^{(g)}$, the design noise $e_{j,k}^{(V)}$, and the stochastic terms $m_{j,k}$ and $\Sigma_{j,k}$, we obtain the following bounds for the $L^p$-risk of $\widehat{V}^{(\mathrm{km})}$ assuming $V \in B^\beta_{r_0, q''}(\mathbb{S}^d)$:
\begin{itemize}
    \item \textit{Regular regime ($p \le T_h$):} For $p \le r_0 (\beta+d/2)$,
    \[
    \mathbb{E}\left[\|\widehat{V}^{(\mathrm{km})} - V\|_{L^p(\mathbb{S}^d)}^p\right] \lesssim \max\left\{ \left(\frac{N}{\log N}\right)^{-\alpha p},\, \left(\frac{N}{\log N}\right)^{-\frac{\beta p}{2\beta+d}} \right\}.
    \]
    \item \textit{Sparse regime ($p > T_h$):} For $p > r_0 (\beta+d/2)$,
    \[
    \mathbb{E}\left[\|\widehat{V}^{(\mathrm{km})} - V\|_{L^p(\mathbb{S}^d)}^p\right] \lesssim \max\left\{ \left(\frac{N}{\log N}\right)^{-\alpha p},\, \left(\frac{N}{\log N}\right)^{-\frac{p\left(\beta - d\left(\frac{1}{r_0} - \frac{1}{p}\right)\right)}{2\beta - d\left(\frac{1}{r_0} - \frac{1}{2}\right)}} \right\}.
    \]
\end{itemize}

\begin{remark}[Deterministic versus stochastic treatment of the mean]
For any $\alpha > 0$, the deterministic exponent $\alpha$ dictates the baseline floor of the bias decay. Projecting $g$ deterministically up to the maximal resolution $J_N$ yields an estimation error that is purely \emph{bias-limited} by the approximation properties of the Besov space, freeing the mean component from any structural statistical variance accumulation at high resolutions.
\end{remark}

\begin{remark}[Fully known mean function and zero-bias limit]
If the mean function $g$ is fully known (equivalent to a zero-mean regression model), the approximation bias vanishes ($\alpha \to \infty$). The estimator $\widehat{V}^{\mathrm{full}}$ (with $g \equiv 0$) achieves the pure adaptive minimax rates of the class $B^\beta_{r_0, q''}(\mathbb{S}^d)$ without any baseline floor:
\[
\sup_{V\in B^\beta_{r_0, q''}(\mathbb{S}^d)}\mathbb{E}\left[\left\| \widehat V^{\mathrm{full}} - V \right\|_{L^p(\mathbb{S}^d)}^p\right] \lesssim \left(\frac{N}{\log N}\right)^{-\mathcal{R}_h(p) p},
\]
where $\mathcal{R}_h(p)$ is the phase-dependent rate governed by $T_h$. In the supremum limit $p \to \infty$, this yields:
\[
\sup_{V\in B^\beta_{r_0, q''}(\mathbb{S}^d)}\mathbb{E}\left[\left\| \widehat V^{\mathrm{full}} - V \right\|_{L^\infty(\mathbb{S}^d)}\right] \lesssim \left(\frac{N}{\log N}\right)^{-\frac{\beta-\frac{d}{r_0}}{2\beta - d\left(\frac{1}{r_0} - \frac{1}{2}\right)}}.
\]
\end{remark}

\section{A simulation study}\label{sec:numerics}

\begin{figure}[!htbp]
  \centering
  \includegraphics[width=14cm]{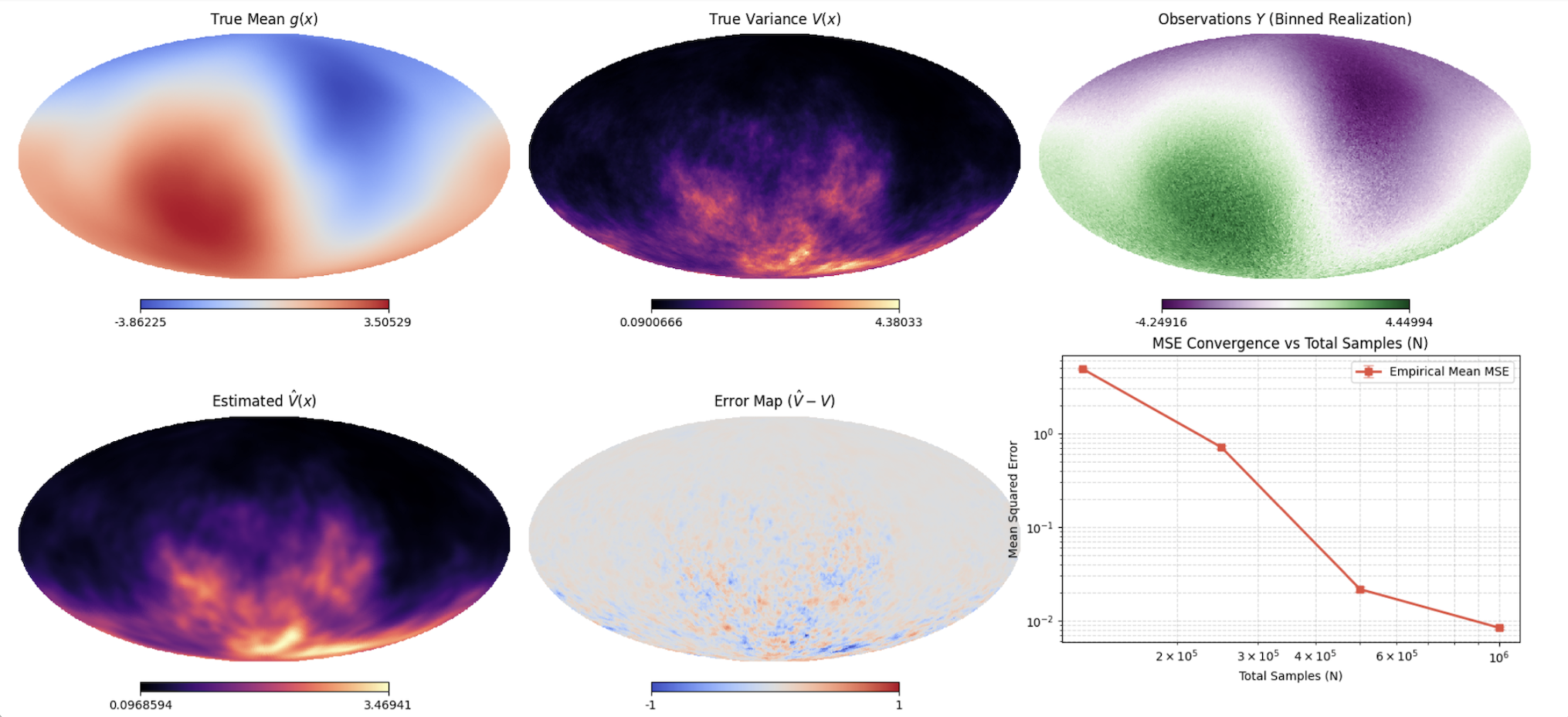}
 \caption{Scenario A (Smooth $g$, Rough $V$): Top row shows the true mean function $g(x)$, the true variance function $V(x)$, and a binned realization of the regression observations with $N=10^6$. Bottom row shows the estimated variance profile $\widehat{V}(x)$, the pointwise error map $(\widehat{V} - V)$, and the log--log empirical MSE convergence curve against the sample size $N$.}
  \label{fig:g-smooth-v-rough}
\end{figure}  

\begin{figure}[!htbp]
  \centering
  \includegraphics[width=14cm]{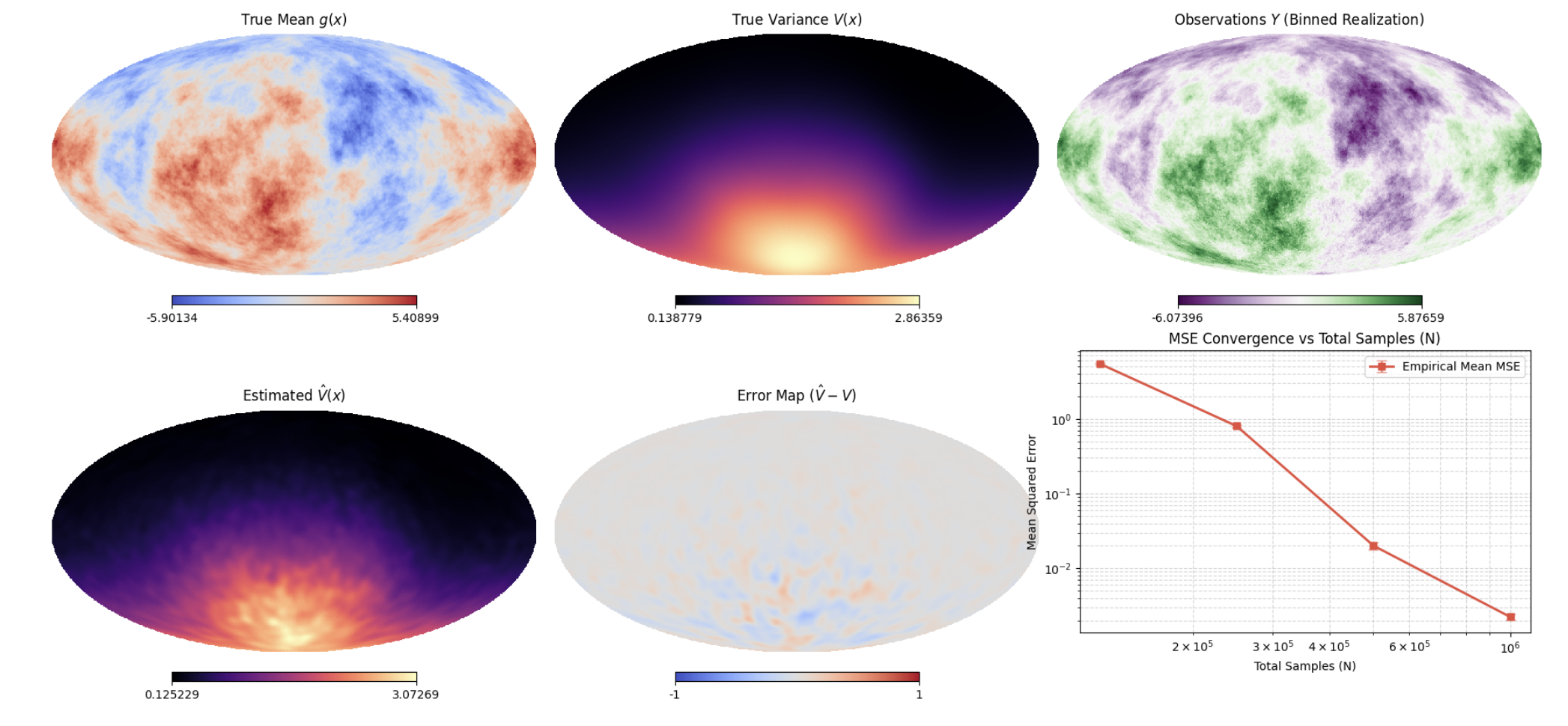}
 \caption{Scenario B (Rough $g$, Smooth $V$): Top row shows the true mean function $g(x)$, the true variance function $V(x)$, and a binned realization of the regression observations with $N=10^6$. Bottom row shows the estimated variance profile $\widehat{V}(x)$, the pointwise error map $(\widehat{V} - V)$, and the log--log empirical MSE convergence curve against the sample size $N$.}
 \label{fig:g-rough-v-smooth}
\end{figure} 

The purpose of this section is to illustrate the performance of the estimator $\widehat{V}$ with a numerical simulation on the sphere $\mathbb{S}^2$.

In order to generate the observations, we use the HEALPix discretization of the sphere at resolution $N_{\rm side} = 64$, which corresponds to a partition of the manifold into a total of $49152$ equal-area pixels. We then draw uniformly from the set of cells to obtain a sample $(X_i)_{i=1,\dots, N}$ approximating uniformly sampled data projected on the HEALPix grid. With this sample, we compute
\begin{equation*}
Y_i=g(X_i)+\sqrt{V(X_i)}\varepsilon_i,\quad i=1,\dots, N,
\end{equation*}
where $\varepsilon_i$ are chosen to be i.i.d. $\mathcal N(0,1)$ random variables.

The functions $g$ and $V$ are constructed as realizations of Gaussian isotropic random fields with the angular power spectrum of the form $C_\ell = \ell^{-2-\gamma}$. As shown in \cite{lang2015isotropic}, such fields belong to the space $\mathcal B^{\gamma/2}_{\infty,\infty}$ of Hölder $\frac{\gamma}{2}$-functions. For $V$, a smooth exponential transformation is applied to such a field in order to obtain a function that is positive everywhere. 

To assess the adaptive behavior of the estimator across the phase-dependent regimes, we consider two distinct operational scenarios:
\begin{itemize}
    \item \textit{Scenario A (Case 1):} Smooth mean $g \in \mathcal B^{2}_{\infty,\infty}(\mathbb{S}^2)$ ($\gamma_g = 4$) and rough variance $V \in \mathcal B^{\frac{1}{2}}_{\infty,\infty}(\mathbb{S}^2)$ ($\gamma_V = 1$). 
    \item \textit{Scenario B (Case 3):} Rough mean $g \in \mathcal B^{\frac{1}{2}}_{\infty,\infty}(\mathbb{S}^2)$ ($\gamma_g = 1$) and smooth variance $V \in \mathcal B^{2}_{\infty,\infty}(\mathbb{S}^2)$ ($\gamma_V = 4$).
\end{itemize}

We evaluate the performance of our thresholded needle-based estimator $\widehat{V}$ for sample sizes ranging from $N = 1.25 \times 10^5$ to $N = 10^6$. The empirical Mean Squared Error (MSE), defined as $\mathbb{E}\|\widehat{V} - V\|_{L^2(\mathbb{S}^2)}^2$, is computed over $100$ independent Monte Carlo replications for each scenario.

The quantitative results of these simulations are summarized in Table~\ref{tab:mse_results}. The reconstructed spherical profiles, pointwise error maps, and empirical rate decays are displayed in Figure \ref{fig:g-smooth-v-rough} (Scenario A) and Figure \ref{fig:g-rough-v-smooth} (Scenario B).

\begin{table}[!htbp]
\centering
\caption{Empirical Mean Squared Error (MSE) of the variance estimator $\widehat{V}$ for Scenarios A and B, averaged over $100$ Monte Carlo runs.}
\label{tab:mse_results}
\begin{tabular}{ccccc}
\toprule
\textbf{Scenario} & \multicolumn{4}{c}{\textbf{Sample Size ($N$)}} \\
\cmidrule{2-5}
 & \textbf{$1.25 \times 10^5$} & \textbf{$2.5 \times 10^5$} & \textbf{$5.0 \times 10^5$} & \textbf{$1.0 \times 10^6$} \\
\midrule
\textbf{Scenario A} (Smooth $g$, Rough $V$) & $4.873$ & $0.712$ & $0.022$ & $0.008$ \\
\textbf{Scenario B} (Rough $g$, Smooth $V$)  & $5.420$ & $0.797$ & $0.020$ & $0.002$ \\
\bottomrule
\end{tabular}
\end{table}

In both scenarios, the visual reconstructions confirm that $\widehat{V}$ successfully captures the true underlying spatial features of the variance. The pointwise error maps demonstrate that the remaining estimation errors are uniformly distributed, showing no critical boundary or peak artifacts. Furthermore, the log-log plots verify that the empirical MSE decays linearly as a function of $N$, illustrating the sharpness of our theoretical minimax convergence rates. Notably, in Scenario B, where the variance is smoother, we observe a faster asymptotic rate of decay ($2.20 \times 10^{-3}$ at $N=10^6$) compared to Scenario A ($8.50 \times 10^{-3}$), in perfect alignment with our theoretical upper bounds.

\section{Proofs}\label{sec:proofs}
This section collects the detailed proofs of all main and auxiliary results presented in the paper.

\begin{proof}[Proof of Proposition \ref{prop:unbiased}]
We establish first the exact expectation of the empirical second-moment needlet coefficients, followed by the uniform bound on their variance.

By definition of the empirical estimator, we have
\[
\begin{split}
\mathbb{E}[\widehat{h}_{j,k}] & = \mathbb{E}\left[\frac{1}{N}\sum_{i=1}^{N} Y_i^2 \psi_{j,k}(X_i)\right] \\
& = \mathbb{E}\left[\frac{1}{N}\sum_{i=1}^{N} \left(g^2(X_i) + 2g(X_i)\sigma(X_i)\varepsilon_i + \sigma^2(X_i)\varepsilon^2_i\right) \psi_{j,k}(X_i)\right].
\end{split}
\]

By leveraging the mutual independence of $\{ \varepsilon_i \}_{i=1}^N$ and $\{ X_i \}_{i=1}^N$, together with the structural noise conditions $\mathbb{E}[\varepsilon_i] = 0$ and $\mathbb{E}[\varepsilon_i^2] = 1$, the joint expectation decomposes as
\[
\begin{split}
\mathbb{E}[\widehat{h}_{j,k}] & = \frac{1}{N}\sum_{i=1}^{N}\mathbb{E}\left[ g^2(X_i) \psi_{j,k}(X_i)\right] + 2 \frac{1}{N}\sum_{i=1}^{N} \mathbb{E}\left[ g(X_i)\sigma(X_i) \psi_{j,k}(X_i)\right] \mathbb{E}\left[ \varepsilon_i \right]\\
&\quad + \frac{1}{N}\sum_{i=1}^{N}\mathbb{E}\left[ \sigma^2(X_i) \psi_{j,k}(X_i)\right] \mathbb{E}\left[ \varepsilon^2_i \right]  \\
& = \int_{\mathbb{S}^d} g^2(x)\psi_{j,k}(x) \,\mathrm{d}x + \int_{\mathbb{S}^d} V(x)\psi_{j,k}(x) \mathrm{d}x \\
& = \left(g^2\right)_{j,k} + v_{j,k},
\end{split}
\]
where the integration is evaluated with respect to the uniform probability measure on $\mathbb{S}^d$. This confirms that $\widehat{h}_{j,k}$ is an unbiased estimator of $h_{j,k}$.

Using the i.i.d.\ nature of the sample pairs, the variance of the empirical sum scales inversely with $N$:
\[
\operatorname{Var}(\widehat h_{j,k}) = \frac{1}{N}\operatorname{Var}\left(Y^2\,\psi_{j,k}(X)\right) \le \frac{1}{N}\mathbb{E}\left[Y^4\psi_{j,k}^2(X)\right],
\]
where $X$ represents a generic covariate vector distributed uniformly on $\mathbb{S}^d$, and $Y = g(X) + \sigma(X)\varepsilon$. Expanding the fourth power of the response variable yields
\[
\begin{split}
\mathbb{E}\left[Y^4\psi_{j,k}^2(X)\right] & = \mathbb{E}\left[\sigma^4(X)\psi_{j,k}^2(X)\right] \mathbb{E}\left[\varepsilon^4\right] + 4\mathbb{E}\left[\sigma^3(X)g(X)\psi_{j,k}^2(X)\right] \mathbb{E}\left[\varepsilon^3\right] \\
& + 6\mathbb{E}\left[\sigma^2(X)g^2(X)\psi_{j,k}^2(X)\right] \mathbb{E}\left[\varepsilon^2\right] + 4\mathbb{E}\left[\sigma(X)g^3(X)\psi_{j,k}^2(X)\right]\mathbb{E}\left[\varepsilon\right] \\
& + \mathbb{E}\left[g^4(X)\psi_{j,k}^2(X)\right].
\end{split}
\]
Recalling that $\mathbb{E}[\varepsilon] = 0$, $\mathbb{E}[\varepsilon_i^2]=1$, and applying the global bounds $\|g\|_{\infty} \le G$, $\|\sigma\|_{\infty} \le S$, alongside the notation $\gamma_m = \mathbb{E}[\varepsilon^m]$, the expression simplifies under the worst-case configuration to
\[
\mathbb{E}\left[Y^4\psi_{j,k}^2(X)\right] \le \left(S^4 \gamma_4 + 4G S^3 \gamma_3 + 6G^2 S^2 + G^4\right) \int_{\mathbb{S}^d} \psi_{j,k}^2(x)\,\mathrm{d}x.
\]
Using \eqref{eqn:norm}, namely, $\int_{\mathbb{S}^d} \psi_{j,k}^2(x)\omega_d^{-1}\,\mathrm{d}x \le C_{\psi,2}^2$, yields
\[
\mathbb{E}\left[Y^4\psi_{j,k}^2(X)\right] \le \left(S^4 \gamma_4 + 4G S^3 \gamma_3 + 6G^2 S^2 + G^4\right) C_{\psi,2}^2 = C_h.
\]
Thus, 
\[\operatorname{Var}(\widehat h_{j,k}) \le \frac{C_h}{N}.\]
\end{proof}
\begin{proof}[Proof of Proposition~\ref{prop:moment}]
We first establish the uniform deviation bound via a three-way decomposition, and then utilize it to derive the centered moment estimates.

\paragraph{\textit{Splitting the deviation event.}}
Let $u > 0$. Expanding $Y_i^2$ via the model equation and subtracting $h_{j,k} = (g^2)_{j,k} + v_{j,k}$, we decompose the error into three distinct processes:
\[
\left\{ \left| \widehat{h}_{j,k} - h_{j,k} \right| \ge u \right\} \subseteq E_1(u) \cup E_2(u) \cup E_3(u),
\]
where
\[
\begin{split}
E_1(u) &= \left\{ \left| \frac{1}{N}\sum_{i=1}^N \left(g^2(X_i)\psi_{j,k}(X_i) - (g^2)_{j,k}\right) \right| \ge \frac{u}{3} \right\}, \\
E_2(u) &= \left\{ \left| \frac{2}{N}\sum_{i=1}^N g(X_i)\sigma(X_i)\varepsilon_i\psi_{j,k}(X_i) \right| \ge \frac{u}{3} \right\}, \\
E_3(u) &= \left\{ \left| \frac{1}{N}\sum_{i=1}^N \left(\sigma^2(X_i)\varepsilon_i^2\psi_{j,k}(X_i) - v_{j,k}\right) \right| \ge \frac{u}{3} \right\}.
\end{split}
\]

\paragraph{\textit{Bounding $E_1(u)$.}}
The variables $Z_{1;i} = g^2(X_i)\psi_{j,k}(X_i)$ are i.i.d.\ and bounded. Under the spherical measure, $\mathbb{E}[Z_{1;i}] = (g^2)_{j,k}$. The variance and $L^\infty$ bounds are given by
\[
\operatorname{Var}(Z_{1;i}) \le \|g\|_\infty^4 \|\psi_{j,k}\|_2^2 \le G^4 C_{\psi,2}^2, \quad \text{and} \quad |Z_{1;i}| \le G^2 C_{\psi,\infty} B^{dj/2}.
\]
Applying Bernstein's inequality to $E_1(u)$ with threshold $u/3$ yields
\[
\Pr(E_1(u)) \le 2\exp\left( -\frac{N u^2}{2\left(9 G^4 C_{\psi,2}^2 + G^2 C_{\psi,\infty} B^{dj/2} u\right)}\right).
\]
Setting $u = \kappa_2 \sqrt{\frac{\log N}{N}}$ and using the resolution scale restriction $B^{dj/2} \le \sqrt{\frac{N}{\log N}}$, we have $B^{dj/2}u \le \kappa_2$. Thus,
\[
\Pr\left(E_1\left(\kappa_2\sqrt{\tfrac{\log N}{N}}\right)\right) \le 2N^{-\delta_{h;1}}, \quad \text{where} \quad \delta_{h;1} = \frac{\kappa_2^2}{2G^2\left(9G^2 C_{\psi,2}^2 + C_{\psi,\infty} \kappa_2\right)}.
\]

\paragraph{\textit{Bounding $E_2(u)$.}}
Consider $E_2(u)$. Conditionally on $X_1,\dots,X_N$, we define the centered linear combination
\[
S_{2;N} := 2\sum_{i=1}^N a_i\varepsilon_i, \qquad a_i := g(X_i)\sigma(X_i)\psi_{j,k}(X_i).
\]
Since the regression errors $\varepsilon_i$ are sub-Gaussian with parameter $\sigma_\varepsilon^2$, there exists an absolute constant $c_{S;2}>0$ such that for any $u>0$ the conditional tail probability satisfies
\[
\Pr\left(\frac{1}{N}|S_{2;N}|\ge \frac{u}{3} \;\middle|\; X_1,\dots,X_N\right) \le 2\exp\!\left(-\frac{c_{S;2}\,u^2}{36\,\sigma_\varepsilon^2\,\frac{1}{N}\sum_{i=1}^N a_i^2}\right).
\]
Let $Z_{2,i} := a_i^2 = \big|g(X_i)\sigma(X_i)\psi_{j,k}(X_i)\big|^2$. The random variables $Z_{2,i}$ are i.i.d.\ and, by applying the uniform needlet bounds $C_{\psi,2}, C_{\psi,4}$, and $C_{\psi,\infty}$, they satisfy:
\[
\begin{aligned}
\mathbb{E}[Z_{2,i}] &\le \|g\|_\infty^2\|\sigma\|_\infty^2\|\psi_{j,k}\|_2^2 \le C^2_{\psi,2} G^2 S^2, \\
|Z_{2,i}| &\le \|g\|_\infty^2\|\sigma\|_\infty^2\|\psi_{j,k}\|_\infty^2 \le C_{\psi,\infty}^2 G^2 S^2 B^{dj}, \\
\operatorname{Var}(Z_{2,i}) &\le \mathbb{E}[Z_{2,i}^2]\le \|g\|_\infty^4\|\sigma\|_\infty^4\|\psi_{j,k}\|_4^4 \le C_{\psi,4}^4 G^4 S^4 B^{dj}.
\end{aligned}
\]
By integrating the conditional tail bound and introducing a variance fluctuation threshold $\upsilon>0$ via a decoupling argument, we obtain
\[
\Pr\left(E_2(u)\right) \le 2\exp\!\left(-\frac{c_{S;2} N u^2}{36\sigma_\varepsilon^2\left(\mathbb{E}[Z_{2,i}]+\upsilon\right)}\right) + \Pr\left(\left|\frac{1}{N}\sum_{i=1}^N (Z_{2,i}-\mathbb{E}[Z_{2,i}])\right|>\upsilon\right).
\]
Applying Bernstein's inequality to the second probability under the established bounds for $\operatorname{Var}(Z_{2,i})$ and $|Z_{2,i}|$ yields
\[
\Pr\left(\left|\frac{1}{N}\sum_{i=1}^N (Z_{2,i}-\mathbb{E}[Z_{2,i}])\right|>\upsilon\right) \le 2\exp\!\left(-\frac{N\upsilon^2}{2B^{dj}\big(C^4_{\psi,4}G^4 S^4 + \tfrac{1}{3}C_{\psi,\infty}^2 G^2 S^2 \upsilon\big)}\right).
\]
Consequently, the unconditional bound for $E_2(u)$ is given by
\[
\begin{split}
\Pr\left(E_2(u)\right) &\le 2\exp\!\left(-\frac{c_{S;2}Nu^2}{36\sigma_\varepsilon^2\left(C_{\psi,2}^2 G^2 S^2 +\upsilon\right)}\right)\\ 
&+ 2\exp\!\left( -\frac{N \upsilon^2}{2B^{dj}\left(C_{\psi,4}^4 G^4 S^4+\frac{1}{3}C_{\psi,\infty}^2 G^2 S^2 \upsilon \right)}\right).
\end{split}
\]
Now we choose the standard resolution scale $u=\kappa_2\sqrt{\frac{\log N}{N}}$ and evaluate at the spatial restriction $B^{dj}\le \frac{N}{\log N}$. By setting the explicit balancing threshold
\[
\upsilon = \kappa_2 \cdot \left( \frac{\sqrt{c_{S;2}} G S}{\sigma_{\varepsilon}} \min \left\{ \frac{C_{\psi,4}^2}{C_{\psi,2}}, C_{\psi,\infty} \right\} \right),
\]
both exponential terms decay polynomially as $N^{-\delta_{h;2}}$. To ensure that $\Pr\left(E_2\left(\kappa_2\sqrt{\frac{\log N}{N}}\right)\right) \le 4N^{-\delta}$ for a targeted allocation exponent $\delta > 0$, the operational calibration threshold $\kappa_{2,2}(\delta)$ must satisfy
\[
\kappa_{2,2}(\delta) = \max \left\{ \frac{6 C_{\psi,2} G S \sigma_\varepsilon}{\sqrt{c_{S;2}}} \sqrt{\delta}, \; \frac{12 C_{\psi,\infty} G S \sigma_\varepsilon}{\sqrt{6 c_{S;2}}}\, \delta \right\},
\]
where the overall decay rate is controlled by
\[
\delta_{h;2} = \min \left( \frac{c_{S;2}\kappa_2^2}{36\sigma_{\varepsilon}^2 C_{\psi,2}^2 G^2 S^2}, \; \frac{\sqrt{6c_{S;2}}\kappa_2}{12C_{\psi,\infty} G S \sigma_{\varepsilon}} \right).
\]

\paragraph{\textit{Bounding $E_3(u)$.}}
Decompose the event into a sub-exponential process and a bounded design process:
\[
\begin{aligned}
\Pr(E_3(u)) &\leq \Pr\left(\left|\frac{1}{N}\sum_{i=1}^N \sigma^2(X_i)\psi_{j,k}(X_i)\eta_i\right| \geq \frac{u}{6}\right) + \Pr\left(\left|\frac{1}{N}\sum_{i=1}^N \sigma^2(X_i)\psi_{j,k}(X_i) - v_{j,k}\right| \geq \frac{u}{6}\right) \\
&=: \Pr(E_{3,1}(u)) + \Pr(E_{3,2}(u)).
\end{aligned}
\]
For $E_{3,1}(u)$, let $a'_i:=\sigma^2(X_i)\psi_{j,k}(X_i)$ and $S_{3;N} = \frac{1}{N} \sum_{i=1}^N (a'_i)^2$. Introducing the concentration event $A_{\upsilon'} = \{ | S_{3;N}- \mathbb{E}[(a'_i)^2] | < \upsilon' \}$ for a fixed threshold $\upsilon'>0$, Bernstein's inequality for sub-exponential random variables conditional on the $\{X_i\}_{i=1}^N$ fields yields:
\[
\Pr(E_{3,1}(u)\cap A_{\upsilon'})\le 2\exp\!\left(-c_{\eta}\min\left( \frac{u^2N}{36\,K_{\varepsilon}^4 (S^4 C^2_{\psi,2}+\upsilon')}, \;\frac{uN}{6\,K^2_{\varepsilon} S^2 C_{\psi,\infty}B^{dj/2}} \right)\right).
\]
On the complementary event $A_{\upsilon'}^C$, setting $Z_{3;i}=(a'_i)^2$, the uniform needlet bounds yield $\operatorname{Var}(Z_{3;i}) \le S^8 C_{\psi,4}^4 B^{dj}$ and $|Z_{3;i}| \le S^4 C_{\psi,\infty}^2 B^{dj}$. Applying Bernstein's inequality on the design space gives:
\[
\Pr\left( A^{C}_{\upsilon'}\right) \le 2\exp\left( -\frac{N (\upsilon')^2}{2B^{dj}\left(S^8 C_{\psi,4}^4 + \frac{1}{3}\!S^4 C_{\psi,\infty}^2\upsilon' \right)} \right).
\]
We now evaluate the thresholds at $u=\kappa_{3,1}\sqrt{\frac{\log N}{N}}$ under the standard resolution restriction $B^{dj}\leq \frac{N}{\log N}$. Choosing the static parameter assignment $\upsilon' = S^4 C_{\psi,2}^2$ simplifies the unconditional tail bound for $E_{3,1}$ to:
\[
\begin{aligned}
\Pr \left(E_{3,1}\left(\kappa_{3,1} \sqrt{\frac{\log N}{N}}\right)\right) &\le 2\exp\!\left(-c_{\eta}\min\!\left( \frac{\kappa_{3,1}^2\log N}{72\,K_{\varepsilon}^4 S^4 C_{\psi,2}^2}, \;\frac{\kappa_{3,1} \log N}{6\,K_{\varepsilon}^2 S^2 C_{\psi,\infty}} \right)\right) \\
& + 2\exp\left( -\frac{\log N \cdot C_{\psi,2}^4}{2\left(C_{\psi,4}^4 + \frac{1}{3}C_{\psi,\infty}^2 C_{\psi,2}^2\right)} \right).
\end{aligned}
\]
For the deterministic design component $E_{3,2}(u)$, Bernstein's inequality evaluated at $u = \kappa_2\sqrt{\frac{\log N}{N}}$ and $B^{dj/2}\le \sqrt{\frac{N}{\log N}}$ yields:
\[
\Pr\left(E_{3,2}\left(\kappa_2\sqrt{\frac{\log N}{N}}\right)\right) \le 2 \exp\left(- \frac{\kappa_2^2 \log N}{4 S^2\left(18 S^2 C_{\psi,2}^2 + \frac{1}{3} C_{\psi,\infty} \kappa_2\right)} \right).
\]
Combining these terms, the polynomial decay rate exponent $\delta_{h;3}$ satisfying $\Pr\left(E_3\left(\kappa_2\sqrt{\frac{\log N}{N}}\right)\right) \le 6 N^{-\delta_{h;3}}$ is given explicitly by the structural minimum:
\[
\delta_{h;3} = \min \left\{ \delta_{3;2}(\kappa_{3,1}), \; \frac{C_{\psi,2}^4}{2 C_{\psi,4}^4 + \frac{2}{3}C_{\psi,\infty}^2 C_{\psi,2}^2}, \; \frac{\kappa_2^2}{4 S^2\left(18 S^2 C_{\psi,2}^2 + \frac{1}{3} C_{\psi,\infty} \kappa_2\right)} \right\},
\]
where the sub-exponential error factor is defined as
\[
\delta_{3;2}(\kappa_{3,1}) = c_{\eta}\min\!\left(\frac{\kappa_{3,1}^2}{72 K_{\varepsilon}^4 S^4 C_{\psi,2}^2},\; \frac{\kappa_{3,1}}{6 K_{\varepsilon}^2 S^2 C_{\psi,\infty}}\right).
\]
For a sufficiently large choice of the calibration threshold $\kappa_{3,1}$, the linear scaling regimes dominate, allowing direct analytic inversion for any allocated exponent $\delta > 0$.

\paragraph{\textit{Union bound for all terms.}}
Recall that for any $u>0$,
\[
\left\{ \left| \widehat{h}_{j,k} - h_{j,k} \right| \ge u \right\} \subseteq E_1(u) \cup E_2(u) \cup E_3(u).
\]
By the union bound, evaluating at the resolution scale $u = \kappa_2\sqrt{\frac{\log N}{N}}$, we obtain
\[
\Pr\left( \left| \widehat{h}_{j,k} - h_{j,k} \right| \ge \kappa_2\sqrt{\frac{\log N}{N}} \right) \le \sum_{m=1}^3 \Pr\left(E_m\left(\kappa_2\sqrt{\frac{\log N}{N}}\right)\right) \le 12 N^{-\delta_{h}},
\]
where $\delta_{h} = \min\{\delta_{h;1},\delta_{h;2},\delta_{h;3}\}$. To ensure that the total tail probability decays at a polynomial rate bounded by a targeted allocation exponent $\delta > 0$, the operational calibration threshold $\kappa_2$ must be chosen sufficiently large.

For a fixed target exponent $\delta_h = \delta$, the exact minimal admissible choices of $\kappa_2$ for each component process are obtained by direct algebraic inversion of the structural exponents derived under the linear scaling regimes:
\[
\begin{aligned}
\kappa_{2,1}(\delta) &= 3 G^2 C_{\psi,\infty} \delta + \sqrt{9 G^4 C_{\psi,\infty}^2 \delta^2 + 18 G^4 C_{\psi,2}^2 \delta}, \\
\kappa_{2,2}(\delta) &= \max \left\{ \frac{6 C_{\psi,2} G S \sigma_\varepsilon}{\sqrt{c_{S;2}}} \sqrt{\delta}, \; \frac{12 C_{\psi,\infty} G S \sigma_\varepsilon}{\sqrt{6 c_{S;2}}}\, \delta \right\}, \\
\kappa_{2,3}(\delta) &= \max \left\{ \frac{6 K_{\varepsilon}^2 S^2 C_{\psi,\infty}}{c_{\eta}}\, \delta, \; 4 S^2 C_{\psi,\infty} \delta \right\}.
\end{aligned}
\]
Consequently, the minimal operational calibration threshold is defined by the maximum of these individual thresholds:
\[
\kappa_2^{\min}(\delta) = \max\{\kappa_{2,1}(\delta),\, \kappa_{2,2}(\delta),\, \kappa_{2,3}(\delta)\}.
\]
Setting $\kappa_2 \ge \kappa_2^{\min}(\delta)$ ensures that all three contributions to the empirical deviation process are of the same exponential order in $N$, with a total probability controlled by
\begin{equation*}\label{eqn:unionbound}
\Pr\left( \left| \widehat{h}_{j,k} - h_{j,k} \right| \ge \kappa_2\sqrt{\frac{\log N}{N}} \right) \le 12 N^{-\delta},
\end{equation*}
which rigorously justifies the rate $u \sim \sqrt{\log N / N}$ in the spatial concentration analysis and shows that the polynomial decay rate $\delta_h$ is directly controlled by $\kappa_2$.

\paragraph{\textit{Centered absolute moment inequalities.}} To show the moment bounds for $r\geq 1$, we proceed similarly to the probability bounds and write
\[
\begin{split}
\mathbb{E} \left[\left| \widehat{h}_{j,k}- h_{j,k}\right|^r \right] &= \mathbb{E} \left[\left| \frac{1}{N}\sum_{i=1}^{N}\left(g^2(X_i) +2 g(X_i)\sigma(X_i)\varepsilon_i + \sigma^2(X_i)\varepsilon^2_i\right) {\psi}_{j,k}(X_i) - h_{j,k}\right|^r \right] \\
&\le 3^{r-1} \left(\mathbb{E} |A_1|^r+\mathbb{E} |A_2|^r + \mathbb{E} |A_3|^r\right),
\end{split}
\]
where
\[
\begin{aligned}
A_1 &= \frac{1}{N}\sum_{i=1}^{N} g^2(X_i){\psi}_{j,k}(X_i) - (g^2)_{j,k}, \\
A_2 &= \frac{2}{N}\sum_{i=1}^{N} g(X_i)\sigma(X_i)\varepsilon_i {\psi}_{j,k}(X_i), \\
A_3 &= \frac{1}{N}\sum_{i=1}^{N} \sigma^2(X_i)\varepsilon^2_i {\psi}_{j,k}(X_i) - v_{j,k}.
\end{aligned}
\]
We analyze these terms one by one. For $A_1$, we have
\[
\mathbb{E} |A_1|^r=\frac{1}{N^r}\mathbb{E} \left| \sum_{i=1}^{N} Z_{1i}-\mathbb{E} Z_{1i}\right|^r.
\]
Since $Z_{1i}$ are i.i.d. random variables with finite moments, the fact that it is of order $N^{-\frac{r}{2}}$ follows directly from Theorem 4 in \cite{BVB65} and the fact that the variance of $Z_{1i}$ is uniformly bounded with respect to $j$ and $k$. The same argument applies also to $A_2$: note that the summands in $A_2$ are centered, and the uniform boundedness of their variance has been shown above: $\mathbb{E}\left[ Z_{2,i} \right] \leq C_{\psi,2}^2 \|g\|_\infty^2 \|\sigma\|_\infty^2$. In $A_3$, we again have i.i.d. random variables $Z_{3i}:=\sigma^2(X_i)\varepsilon^2_i {\psi}_{j,k}(X_i)$ with expectation $v_{j,k}$ (since $\varepsilon_i$ all have unit variance) and variance
\[
\mathrm{Var}(Z_{3i})\leq \int_{\mathbb{S}^d} \sigma^4(x)\vert {\psi}_{j,k}(x)\vert ^2 \,\mathrm{d}x \cdot \mathbb{E} \varepsilon^4_1 \lesssim \|\sigma\|_\infty^4,
\]
which allows us to apply Theorem 4 from \cite{BVB65} again and obtain the same bound.

We focus now on \eqref{eqn:expectsup}. Observe that
\[
\mathbb{E}\left[\sup_{k=1,\ldots,K_j}\left|\widehat{h}_{j,k}-h_{j,k}\right|^r\right] \le \int_{0}^{\infty} r u^{r-1} \Pr\left(\sup_{k=1,\ldots,K_j} \left| \widehat {h}_{j,k}-{h}_{j,k}\right| \geq u \right) \,\mathrm{d}u.
\]
Also, for $B^{dj}\le N$,
\[
\Pr(E_1(u)) \le 2\exp\left( -\frac{N u^2}{18\|g\|_\infty^4 C^2_{\psi,2}}\right) + 2\exp\!\left( -\frac{\sqrt{N} u} {12\|g\|_\infty^2 C_{\psi,\infty}} \right),
\]
while for sufficiently large $\upsilon$
\[\begin{split}
\Pr\left(E_2(u)\right) &\le 2\exp\!\left(-\frac{c_{S;2}Nu^2}{72\sigma_\varepsilon^2C_{\psi,2}^2\|g\|_\infty^2\|\sigma\|_\infty^2}\right) +2\exp\! \left( -\frac{c_{S;2} Nu^2}{72\sigma_\varepsilon^2\upsilon}\right) \\ &+ 2\exp\!\left(-\frac{3\upsilon}{4C_{\psi,\infty}^2\|g\|_\infty^2\|\sigma\|_\infty^2}\right).\end{split}
\]
To balance the second and the third exponential terms, we choose $\upsilon>0$ so that the two exponents are equal:
\[
\frac{c_{S;2}\,N u^2}{72\,\sigma_\varepsilon^2\,\upsilon} = \frac{3\,\upsilon}{4 C_{\psi,\infty}^2 \|g\|_\infty^2 \|\sigma\|_\infty^2}.
\]
Solving for $\upsilon$ gives
\[
\upsilon = u\sqrt{N}\,\frac{\sqrt{c_{S;2}/54}}{\sigma_\varepsilon\,C_{\psi,\infty}\,\|g\|_\infty\,\|\sigma\|_\infty}.
\]
With this choice, both terms are of the same order and we obtain the simplified bound
\[
2\exp\left(-\frac{c_{S;2}\,N u^2}{72\,\sigma_\varepsilon^2\,\upsilon}\right) + 2\exp\left(-\frac{3\,\upsilon}{4 C_{\psi,\infty}^2 \|g\|_\infty^2 \|\sigma\|_\infty^2}\right) \le 4 \exp\left(-u\sqrt{N}\,\frac{\sqrt{c_{S;2}/54}}{\sigma_\varepsilon\,C_{\psi,\infty}\,\|g\|_\infty\,\|\sigma\|_\infty}\right).
\]
Thus, in this case
\[
\Pr\left(E_2(u)\right) \le 2\exp\!\left(-\frac{c_{S;2}Nu^2}{72\sigma_{\varepsilon}^2C_{\psi,2}^2\|g\|_\infty^2\|\sigma\|_\infty^2}\right) + 4 \exp\left(-u\sqrt{N}\,\frac{\sqrt{c_{S;2}/54}}{\sigma_\varepsilon\,C_{\psi,\infty}\,\|g\|_\infty\,\|\sigma\|_\infty}\right).
\]
Under the assumption $B^{dj}\le N$ and that $\upsilon'$ is large, we can remove the $\min$ in the first exponential of $E_{3,1}(u)$ and write
\[
\Pr(E_{3,1}(u)) \lesssim 2\exp\Big(-c_\eta\,\frac{N u^2}{36 K_\varepsilon^4 \upsilon'}\Big) + 2\exp\Big(-\frac{N (\upsilon')^2}{2 (\|\sigma\|_\infty^8 C_{\psi,4}^4 B^{dj} + \frac{1}{3} \|\sigma\|_\infty^4 C_{\psi,\infty}^2 B^{dj}\upsilon')}\Big).
\]
Choosing $\upsilon'$ to balance the two terms by equating the leading behavior of the exponents:
\[
c_\eta\,\frac{N u^2}{36 K_\varepsilon^4 \upsilon'} = \frac{N (\upsilon')^2}{2 \|\sigma\|_\infty^8 C_{\psi,4}^4 B^{dj}},
\]
we obtain
\[
\upsilon' = \left[ \frac{2 c_\eta \|\sigma\|_\infty^8 C_{\psi,4}^4}{36 K_\varepsilon^4}\, u^2 B^{dj} \right]^{1/3}.
\]
Using $B^{dj}\le N$, this gives the leading-order bound
\[
\Pr(E_{3,1}(u)) \lesssim 4 \exp\left( - c\, u^{2/3} N^{1/3} \right),
\]
with
\[
c = \frac{1}{K_\varepsilon^{4/3}} \left( \frac{c_\eta \|\sigma\|_\infty^8 C_{\psi,4}^4}{18} \right)^{1/3}.
\]
This shows that, for large $\upsilon'$, the tail probability decays sub-exponentially with $u^{2/3} N^{1/3}$.
Finally,
\[
\Pr \left(E_{3,2}(u)\right) \leq 2 \exp\left(- \frac{N u^2}{144 \|\sigma\|_\infty^4 C^2_{\psi,2}} \right) + 2 \exp\left(- \frac{\sqrt{N} u}{ \frac{144}{3} \|\sigma\|_\infty^2 C_{\psi,\infty}} \right).
\]
Combining all these terms, we have
\[
\Pr\left(\sup_{k=1,\ldots,K_j} \left| \widehat {h}_{j,k}-{h}_{j,k}\right| \geq \!u\! \right) \!\leq\! C_1\exp\left(-c_1 Nu^2\right)+C_2\exp\left(-c_2 \sqrt{N}u\right)+C_3\exp\left( - c_3 u^{2/3} N^{1/3} \right).
\]
Following \cite{DGM12}, fixing a proper constant $A^\prime$, we separate the integration region:
\[
\begin{aligned}
\mathbb{E}\left[\sup_{k=1,\ldots,K_j}\left|\widehat{h}_{j,k}-h_{j,k}\right|^r\right] &\leq \int_{u\leq A^\prime \frac{j+1}{\sqrt{N}}} r u^{r-1} \,\mathrm{d}u + \int_{u> A^\prime \frac{j+1}{\sqrt{N}}} \widetilde{C} r u^{r-1}\exp(-c_3 u^{2/3}N^{1/3}) \,\mathrm{d}u \\
&\lesssim \left(j+1\right)^r N^{-\frac{r}{2}},
\end{aligned}
\]
as claimed.
\end{proof}

\begin{proof}[Proof of Lemma \ref{lemma:unbiasedgsquared}]
Let us preliminarily consider \eqref{eqn:bias1}. On the one hand, we have that
\begin{equation*}
\widehat{\left(g^2\right)}_{j,k}
=\sum_{j_1=0}^{J_N^\prime-1}\sum_{j_2=0}^{J_N^\prime-1}\sum_{k_1=1}^{K_{j_1}}\sum_{k_2=1}^{K_{j_2}}
\widehat g^{(1)}_{j_1,k_1} \widehat g^{(2)}_{j_2,k_2}
\int_{\mathbb{S}^d}\psi_{j_1,k_1}(x)\psi_{j_2,k_2}(x)\psi_{j,k}(x)\,\mathrm{d}x.
\end{equation*}
On the other hand, taking expectations and using the independence of the two subsamples yield
\[
\mathbb{E}\left[\widehat g^{(1)}_{j_1,k_1}\widehat g^{(2)}_{j_2,k_2}\right]
=\mathbb{E}\left[\widehat g^{(1)}_{j_1,k_1}\right] \mathbb{E}\left[\widehat g^{(2)}_{j_2,k_2}\right].
\]
For $m=1,2$, recalling that $\mathbb{E}\left[\widehat g^{(m)}_{j,k}\right]=g_{j,k}$, and denoting the thresholding indicators by
\[
E^{(m)}_{j,k}=\mathbf{1}\left\{\left| \widehat g^{(m)}_{j,k}\right| < \kappa_g \sqrt{\tfrac{\log N^\prime}{N^\prime}}\right\},
\]
we obtain the expansion
\[
\mathbb{E}\left[\widehat{(g^2)}_{j,k}\right] = (g^2)_{j,k} - T^{\mathrm{thres}}_{j,k} - T^{\mathrm{trunc}}_{j,k} - T^{\mathrm{cross}}_{j,k},
\]
where the thresholding bias is given by
\[
T^{\mathrm{thres}}_{j,k} = \mathbb{E} \left[\sum_{j_1=0}^{J_{N^\prime}-1}\sum_{j_2=0}^{J_{N^\prime}-1}\sum_{k_1,k_2} g_{j_1,k_1}\,g_{j_2,k_2} E^{(1)}_{j_1,k_1}E^{(2)}_{j_2,k_2} \int_{\mathbb{S}^d}\psi_{j_1,k_1}(x)\psi_{j_2,k_2}(x)\psi_{j,k}(x)\,\mathrm{d}x \right],
\]
the truncation bias accounts for omitted scales above the cutoff $J_{N^\prime}$:
\[
T^{\mathrm{trunc}}_{j,k} = \sum_{j_1\ge J_{N^\prime}}\sum_{j_2\ge J_{N^\prime}} \sum_{k_1,k_2} g_{j_1,k_1}\,g_{j_2,k_2} \int_{\mathbb{S}^d}\psi_{j_1,k_1}(x)\psi_{j_2,k_2}(x)\psi_{j,k}(x)\,\mathrm{d}x,
\]
and the cross-term bias is given by
\[
T^{\mathrm{cross}}_{j,k} = 2\sum_{j_1=0}^{J_{N^\prime}-1}\sum_{j_2\ge J_{N^\prime}} \sum_{k_1,k_2} g_{j_1,k_1}\,g_{j_2,k_2} \int_{\mathbb{S}^d}\psi_{j_1,k_1}(x)\psi_{j_2,k_2}(x)\psi_{j,k}(x)\,\mathrm{d}x.
\]

As far as the thresholding bias $T^{\mathrm{thres}}_{j,k}$ is concerned, first observe that from Remark \ref{rem:Js}, for $m=1,2$, if $j_m\in\left\{J_s,\ldots, J_{N^\prime}-1\right\}$, then $E_{j_m,k_m}^{(m)} \equiv \emptyset$. Combining this consideration with Fubini's theorem, we can write:
\[
\begin{aligned}
\left| T^{\mathrm{thres}}_{j,k} \right| &= \left| \int_{\mathbb{S}^d} \mathbb{E} \left[\sum_{j_1=0}^{J_{s}-1}\sum_{j_2=0}^{J_{s}-1}\sum_{k_1,k_2} g_{j_1,k_1}\,g_{j_2,k_2} E^{(1)}_{j_1,k_1}E^{(2)}_{j_2,k_2}\psi_{j_1,k_1}(x)\psi_{j_2,k_2}(x)\right]\psi_{j,k}(x)\,\mathrm{d}x \right| \\
&\le \int_{\mathbb{S}^d} \left| \mathbb{E} \left[\sum_{j_1=0}^{J_{s}-1}\sum_{j_2=0}^{J_{s}-1}\sum_{k_1,k_2} g_{j_1,k_1}\,g_{j_2,k_2} E^{(1)}_{j_1,k_1}E^{(2)}_{j_2,k_2}\psi_{j_1,k_1}(x)\psi_{j_2,k_2}(x)\right] \right| \big| \psi_{j,k}(x)\big|\,\mathrm{d}x.
\end{aligned}
\]
For $m=1,2$, we exploit Besov embeddings \eqref{eqn:embedding} and stochastic deviation properties \eqref{eq:probboundg} to establish the following bound uniformly in $x \in \mathbb{S}^d$:
\[
\begin{aligned}
& \left| \mathbb{E}\left[\sum_{j_m=0}^{J_{N^\prime}-1}\sum_{k_m=1}^{K_{j_m}} g_{j_m,k_m}\psi_{j_m,k_m}(x) E^{(m)}_{j_m,k_m}\right]\right| \\
& \qquad\le \sum_{j_m=0}^{J_{\alpha}-1}\sum_{k_m=1}^{K_{j_m}} \big| g_{j_m,k_m}\big| \big|\psi_{j_m,k_m}(x)\big| \mathbb{E} \left[ \mathbf{1}\left\{\left| \widehat g^{(m)}_{j_m,k_m}\right| < \kappa_g \sqrt{\tfrac{\log N^\prime}{N^\prime}}\right\} \mathbf{1}\left\{\left| g_{j_m,k_m}\right| \ge 2\kappa_g \sqrt{\tfrac{\log N^\prime}{N^\prime}}\right\} \right] \\
& \qquad\quad+ \sum_{j_m=0}^{J_{\alpha}-1}\sum_{k_m=1}^{K_{j_m}} \big| g_{j_m,k_m}\big| \big|\psi_{j_m,k_m}(x)\big| \mathbb{E} \left[ \mathbf{1}\left\{\left| \widehat g^{(m)}_{j_m,k_m}\right| < \kappa_g \sqrt{\tfrac{\log N^\prime}{N^\prime}}\right\} \mathbf{1}\left\{\left|g_{j_m,k_m}\right| < 2\kappa_g \sqrt{\tfrac{\log N^\prime}{N^\prime}}\right\} \right] \\
& \qquad\le \sum_{j_m=0}^{J_{\alpha}-1}\sum_{k_m=1}^{K_{j_m}} \big| g_{j_m,k_m}\big| \big|\psi_{j_m,k_m}(x)\big| \Pr \left(\left|\widehat g^{(m)}_{j_m,k_m} - g_{j_m,k_m} \right| > \kappa_g \sqrt{\tfrac{\log N^\prime}{N^\prime}}\right) \\
& \qquad\quad + \sum_{j_m=0}^{J_{\alpha}-1}\sum_{k_m=1}^{K_{j_m}} \big| g_{j_m,k_m}\big| \big|\psi_{j_m,k_m}(x)\big| \mathbf{1}\left\{\left| g_{j_m,k_m}\right| < 2\kappa_g \sqrt{\tfrac{\log N^\prime}{N^\prime}}\right\} \\
& \qquad=: T_1^{\mathrm{thres}} + T_2^{\mathrm{thres}}.
\end{aligned}
\]

Consider $T_1^{\mathrm{thres}}$. First, we use \eqref{eq:probboundg} to get:
\[
T_1^{\mathrm{thres}} \le \left(N^\prime\right)^{-\gamma_g} \sum_{j_m=0}^{J_{\alpha}-1}\sum_{k_m=1}^{K_{j_m}} \big| g_{j_m,k_m}\big| \big|\psi_{j_m,k_m}(x)\big|.
\]
Now, using the sup-norm bound \eqref{eqn:norm} and the Besov embedding property \eqref{eqn:embedding}, we find:
\[
\begin{aligned}
\sum_{k_m=1}^{K_{j_m}} \big| g_{j_m,k_m}\big| \big| \psi_{j_m,k_m}(x)\big| &\le C B^{dj_m/2} \sum_{k_m=1}^{K_{j_m}} |g_{j_m,k_m}| \\
&\le C B^{dj_m/2} K_{j_m}^{1-1/\rho} \left( \sum_{k_m=1}^{K_{j_m}} |g_{j_m,k_m}|^\rho \right)^{1/\rho}.
\end{aligned}
\]
Applying the Besov bound on the $\ell^\rho$ norm of the coefficients yields:
\[
\sum_{k_m=1}^{K_{j_m}} |g_{j_m,k_m}|\,|\psi_{j_m,k_m}(x)| \lesssim \|g\|_{B^\alpha_{\rho,q^\prime}} K_{j_m}^{1-1/\rho} B^{j_m(d/2-\alpha)}.
\]
Observe now that summing over the active resolution scales gives:
\[
\begin{aligned}
\sum_{j_m=0}^{J_\alpha-1} \sum_{k_m=1}^{K_{j_m}} |g_{j_m,k_m}|\,|\psi_{j_m,k_m}(x)| &\lesssim \|g\|_{B^\alpha_{\rho,q^\prime}} \sum_{j_m=0}^{J_\alpha-1} B^{j_m(3d/2 - \alpha - d/\rho)} \\
&\lesssim \|g\|_{B^\alpha_{\rho,q^\prime}} B^{J_\alpha(3d/2 - \alpha - d/\rho)}.
\end{aligned}
\]
Using \eqref{eqn:js} and $N^\prime=N/2$ leads to:
\[
T_1^{\mathrm{thres}} \lesssim N^{-\delta_\gamma} \left(\frac{N}{\log N}\right)^{\frac{3d/2 - \alpha - d/\rho}{2(\alpha+1)}}.
\]
Now note that for the second components $T_2^{\mathrm{thres}}$:
\[
\begin{aligned}
T_2^{\mathrm{thres}} &\lesssim \sum_{j_m=0}^{J_\alpha-1} B^{dj_m/2} \cdot 2\kappa_g \left(\frac{\log N}{N}\right)^{1/2} \\
&\lesssim 2 \kappa_g \left(\frac{\log N}{N}\right)^{1/2} \sum_{j_m=0}^{J_\alpha-1} B^{dj_m/2} \\
&\lesssim 2 \kappa_g \left(\frac{\log N}{N}\right)^{1/2} B^{dJ_\alpha/2} \\
&= \left(\frac{\log N}{N}\right)^{\frac{\alpha}{2(\alpha+1)}}.
\end{aligned}
\]
Combining these definitions yields:
\[
\begin{aligned}
\left|T^{\mathrm{thres}}_{j,k} \right| &\lesssim (T_1^{\mathrm{thres}} + T_2^{\mathrm{thres}})^2\left\| \psi_{j,k}\right\|_{1} \\
&\lesssim \left(N^{-\delta_\gamma} \left(\frac{N}{\log N}\right)^{\frac{3d/2 - \alpha - d/\rho}{2(\alpha+1)}} + \left(\frac{N}{\log N}\right)^{-\frac{\alpha}{2(\alpha+1)}} \right)^2 B^{-dj/2}.
\end{aligned}
\]
Simple algebraic manipulations yield:
\[
T^{\mathrm{thres}}_{j,k} \lesssim
\begin{cases}
N^{-2\delta_\gamma} \left(\frac{N}{\log N}\right)^{\frac{3d/2-\alpha-d/\rho}{\alpha+1}} B^{-dj/2}, & \text{if } \delta_\gamma < \frac{3d/2 - d/\rho}{2(\alpha+1)}, \\
\left(\frac{N}{\log N}\right)^{-\frac{\alpha}{\alpha+1}} B^{-dj/2}, & \text{if } \delta_\gamma \ge \frac{3d/2 - d/\rho}{2(\alpha+1)}.
\end{cases}
\]
Fixing $\delta_\gamma \ge \frac{3d/2 - d/\rho}{2(\alpha+1)}$ gives:
\[
T^{\mathrm{thres}}_{j,k} \lesssim \left(\frac{N}{\log N}\right)^{-\frac{\alpha}{\alpha+1}} B^{-dj/2}.
\]
Finally, since $\alpha>d/\rho$, we obtain the claimed result.

Now, let us bound the truncation term. Observe that:
\[
\begin{aligned}
\left| T^{\mathrm{trunc}}_{j,k} \right| &=\left|\int_{\mathbb{S}^d} \sum_{j_1\geq J_{N^{\prime}}}\sum_{j_2\geq J_{N^{\prime}}} \sum_{k_1,k_2} g_{j_1,k_1}\psi_{j_1,k_1}(x) g_{j_2,k_2}\psi_{j_2,k_2}(x)\psi_{j,k}(x)\,\mathrm{d}x\right| \\
&\le \int_{\mathbb{S}^d} \left| \sum_{j_1\geq J_{N^{\prime}}} \sum_{k_1} g_{j_1,k_1}\psi_{j_1,k_1}(x)\right| \left| \sum_{j_2\geq J_{N^{\prime}}} \sum_{k_2} g_{j_2,k_2}\psi_{j_2,k_2}(x)\right| \big| \psi_{j,k}(x)\big|\,\mathrm{d}x.
\end{aligned}
\]
Notice that for each coordinate index $m=1,2$:
\[
\begin{aligned}
\left| \sum_{j_m\geq J_{N^{\prime}}} \sum_{k_m} g_{j_m,k_m}\psi_{j_m,k_m}(x)\right| &\le \sum_{j_m\geq J_{N^{\prime}}} \sum_{k_m} \big| g_{j_m,k_m}\big| \big|\psi_{j_m,k_m}(x)\big| \\
&\le \sum_{j_m\geq J_{N^{\prime}}} C B^{dj_m/2} \sum_{k_m} \big| g_{j_m,k_m}\big| \\
&\le \sum_{j_m\geq J_{N^{\prime}}} C B^{dj_m/2} K_{j_m}^{1-\frac{1}{\rho}} \left( \sum_{k_m} \big| g_{j_m,k_m}\big|^\rho \right)^{1/\rho} \\
&\le \sum_{j_m\geq J_{N^{\prime}}} C B^{j_m\left(d/2 + d - d/\rho - \alpha \right)} \\
&\lesssim B^{J_{N^\prime}\left(3d/2 - d/\rho - \alpha\right)} \\
&\lesssim \left( \frac{N}{\log N}\right)^{\frac{3d/2 - d/\rho - \alpha}{2}}.
\end{aligned}
\]
Thus, we have that:
\[
\left| T^{\mathrm{trunc}}_{j,k} \right| \lesssim \left( \frac{N}{\log N}\right)^{3d/2 - \alpha - d/\rho} \left\| \psi_{j,k}\right\|_1.
\]
Analogous reasoning for the cross terms yields:
\[
\begin{aligned}
\left| T^{\mathrm{cross}}_{j,k} \right| &\le 2\int_{\mathbb{S}^d} \left(\sum_{j_1=0}^{J_{N^\prime}-1}\sum_{k_1} \big| g_{j_1,k_1}\big| \big| \psi_{j_1,k_1}(x)\big| \right) \left(\sum_{j_2\geq J_{N^\prime}}\sum_{k_2} \big| g_{j_2,k_2}\big| \big| \psi_{j_2,k_2}(x)\big|\right) \big| \psi_{j,k}(x)\big|\,\mathrm{d}x \\
&\lesssim \left(\sum_{j_1=0}^{J_{\alpha}-1} B^{j_1\left(3d/2-d/\rho \right)} \sum_{k_1} \big| g_{j_1,k_1}\big|^\rho\right) \left(\sum_{j_2\geq J_{N^\prime}} B^{j_2\left(3d/2-d/\rho \right)} \sum_{k_2} \big| g_{j_2,k_2}\big|^\rho \right) \left\|\psi_{j,k}\right\|_1 \\
&\lesssim B^{J_\alpha\left(d - d/\rho - \alpha\right)} B^{J_{N^\prime}\left(d - d/\rho - \alpha\right)} B^{-dj/2} \\
&\lesssim \left(\frac{N}{\log N}\right)^{\frac{d - d/\rho - \alpha}{2(\alpha+1)}} \left(\frac{N}{\log N}\right)^{\frac{3d/2 - d/\rho - \alpha}{2}} B^{-dj/2} \\
&\lesssim \left(\frac{N}{\log N}\right)^{\frac{3d/2 - \alpha^2 - \frac{\alpha+2}{\rho}}{\alpha+1}} B^{-dj/2}.
\end{aligned}
\]
Finally, since $\alpha>d/\rho$, the dominant term is $\left( N /\log N\right)^{-\frac{\alpha}{\alpha+1}}$ and we obtain the claimed parameter bounds.

Let us now consider the global function error $\widehat{\left(g^2\right)}(x)-g^2(x)$ and observe that:
\[
\left| \mathbb{E} \left[\widehat{\left(g^2\right)}(x)-g^2(x) \right] \right| \le \sum_{j=0}^{J_{N^\prime}-1}\sum_{k=1}^{K_j} \left|\mathbb{E}\left[\widehat{\left(g^2\right)}_{j,k}-\left(g^2\right)_{j,k}\right]\right| \big| \psi_{j,k}(x)\big| + \sum_{j\geq J_{N^\prime}} \sum_{k=1}^{K_j} \big| \left(g^2\right)_{j,k}\big| \big| \psi_{j,k}(x)\big|.
\]
Concerning the first term, it holds that:
\[
\begin{aligned}
\sum_{j=0}^{J_{N^\prime}-1}\sum_{k=1}^{K_j} \left|\mathbb{E}\left[\widehat{\left(g^2\right)}_{j,k}-\left(g^2\right)_{j,k}\right]\right| \big| \psi_{j,k}(x)\big| &\lesssim \left(\frac{N}{\log N}\right)^{-\frac{\alpha}{\alpha+1}} \sum_{j=0}^{J_{N^\prime}-1}\sum_{k=1}^{K_j} B^{-dj/2}\left\| \psi_{j,k}\right\|_\infty \\
&\lesssim \left(\frac{N}{\log N}\right)^{-\frac{\alpha}{\alpha+1}+\frac{1}{2}} \\
&\lesssim \left(\frac{N}{\log N}\right)^{\frac{1-\alpha}{2(\alpha+1)}}.
\end{aligned}
\]
Regarding the second term, combining Besov embeddings and Jackson-type inequalities we get:
\[
\begin{aligned}
\sum_{j\geq J_{N^\prime}} \sum_{k=1}^{K_j} \big| \left(g^2\right)_{j,k}\big| \big| \psi_{j,k}(x)\big| &\le \sum_{j\geq J_{N^\prime}} C B^{dj/2} \sum_{k=1}^{K_j} \big| \left(g^2\right)_{j,k}\big| \\
&\le \sum_{j\geq J_{N^\prime}} C B^{dj/2} K_{j}^{1-\frac{1}{\rho}}\sum_{k=1}^{K_j} \big| \left(g^2\right)_{j,k}\big|^\rho \\
&\le \sum_{j\geq J_{N^\prime}} C B^{j\left(3d/2-\alpha-\frac{d}{\rho} \right)} \\
&\lesssim B^{J_{N^\prime}\left(3d/2-\alpha-\frac{d}{\rho}\right)} \\
&\lesssim \left(\frac{N}{\log N}\right)^{\frac{3d/2-\alpha}{2}-\frac{d}{\rho}},
\end{aligned}
\]
as claimed.
\end{proof}

\begin{proof}[Proof of Lemma \ref{lemma:quadraticrisk}]
We show that the moments of $\widehat{g^2}-g^2$ are bounded by the moments of $\widehat{g}^{(m)}-g$, $m=1,2$. The result then follows from the established risk rates of the density estimator \eqref{eq: mean-est}.

To simplify the notation, we work on the product probability space $(\Omega\times\mathbb{S}^d, \mathcal{F}\otimes\mathcal{B}(\mathbb{S}^d), \mathbb{P}\otimes \mathrm{d}x)$. For $1\le p<\infty$, we define the space
\[
L^p(\Omega\times\mathbb{S}^d) := \left\{F:\Omega\times\mathbb{S}^d\to\mathbb{R} \text{ measurable} : \|F\|_{L^p(\Omega\times\mathbb{S}^d)}^p := \mathbb{E}\int_{\mathbb{S}^d}|F(\omega,x)|^p\,\mathrm{d}x < \infty\right\},
\]
equipped with the norm
\[
\left\|F\right\|_{L^p(\Omega\times\mathbb{S}^d)} = \left(\int_{\mathbb{S}^d}\mathbb{E}\big|F(\cdot,x)\big|^p\,\mathrm{d}x\right)^{1/p}.
\]
Starting from the algebraic cross-fitting decomposition
\[
\widehat{g^2}-g^2 = \widehat g^{(1)}(\widehat g^{(2)}-g) + g(\widehat g^{(1)}-g),
\]
we apply the triangle inequality on $L^p(\Omega\times\mathbb{S}^d)$ to obtain
\[
\|\widehat{g^2}-g^2\|_{L^p(\Omega\times\mathbb{S}^d)} \le \left\|\widehat g^{(1)}(\widehat g^{(2)}-g)\right\|_{L^p(\Omega\times\mathbb{S}^d)} + \left\|g(\widehat g^{(1)}-g)\right\|_{L^p(\Omega\times\mathbb{S}^d)}.
\]

For the first term, applying the definition of the product norm, we can pull out the uniform norm of the first factor:
\[
\begin{split}
\left\|\widehat g^{(1)}(\widehat g^{(2)}-g)\right\|_{L^p(\Omega\times\mathbb{S}^d)}&  = \left( \mathbb{E}\int_{\mathbb{S}^d} \left|\widehat g^{(1)}(x)\right|^p \left|\widehat g^{(2)}(x)-g(x)\right|^p \,\mathrm{d}x \right)^{1/p} \\ & \lesssim \left( \mathbb{E}\left[ \left\|\widehat g^{(1)}\right\|_\infty^p \|\widehat g^{(2)}-g\|_{L^p(\mathbb{S}^d)}^p \right] \right)^{1/p}.
    \end{split}
\]
By construction, the cross-fitted sub-samples $\widehat g^{(1)}$ and $\widehat g^{(2)}$ are mutually independent. This independence allows us to factorize the joint expectation exactly:
\[
\mathbb{E}\left[ \left\|\widehat g^{(1)}\right\|_\infty^p \|\widehat g^{(2)}-g\|_{L^p(\mathbb{S}^d)}^p \right] = \mathbb{E}\left[\left\|\widehat g^{(1)}\right\|_\infty^p\right] \mathbb{E}\left[\|\widehat g^{(2)}-g\|_{L^p(\mathbb{S}^d)}^p\right].
\]
Using the global bound $\|g\|_\infty \le G$, the expectation of the uniform norm is bounded via the triangle inequality as:
\[
\mathbb{E}\left\|\widehat g^{(1)}\right\|_\infty^p \le \mathbb{E}\left(G + \left\|\widehat g^{(1)}-g\right\|_\infty\right)^p \le 2^{p-1} \left(G^p + \mathbb{E}\left\|\widehat g^{(1)}-g\right\|_\infty^p\right).
\]
Since the sub-samples are identically distributed, combining these bounds yields:
\[
\mathbb{E}\left\|\widehat{g^2}-g^2\right\|_{L^p(\mathbb{S}^d)}^p \le 2^{p-1}\left(2^{p-1}\left(G^p + \mathbb{E}\left\|\widehat g-g\right\|_\infty^p\right) + G^p\right) \mathbb{E}\|\widehat g-g\|_{L^p(\mathbb{S}^d)}^p.
\]
Whenever $\mathbb{E}\left\|\widehat g-g\right\|_\infty^p = o(1)$ as $N \to \infty$, we obtain the simplified non-asymptotic bound
\[
\left(\mathbb{E}\left\|\widehat{g^2}-g^2\right\|_{L^p(\mathbb{S}^d)}^p\right)^{1/p} \le (2^{\frac{2p-2}{p}}G + o(1)) \left(\mathbb{E}\|\widehat g-g\|_{L^p(\mathbb{S}^d)}^p\right)^{1/p},
\]
which proves that the $L^p$-risk of $\widehat{g^2}$ scales at the same rate as $\widehat{g}$.
\end{proof}

\begin{proof}[Proof of Theorem \ref{thm: h-est}]
The proof follows closely the standard strategy for establishing upper bounds of $L^p$-risks as developed in \cite{BKMP09,Dur15,Dur16,DGM12}. Observe that
\begin{align*}
\mathbb{E} \left[ \left\| \widehat{h} - h \right\|_p^p\right] &= \mathbb{E} \left[ \left\| \sum_{j=0}^{J_N-1}\sum_{k=1}^{K_j}\left(\tilde h_{j,k}-h_{j,k}\right)\psi_{j,k} - \sum_{j\geq J_N}\sum_{k=1}^{K_j}h_{j,k}\psi_{j,k} \right\|_p^p\right] \\
&\leq 2^{p-1} \left( \mathbb{E} \left[ \left\| \sum_{j=0}^{J_N-1}\sum_{k=1}^{K_j}\left(\tilde h_{j,k}-h_{j,k}\right)\psi_{j,k} \right\|_p^p\right] + \left\| \sum_{j\geq J_N}\sum_{k=1}^{K_j}h_{j,k}\psi_{j,k} \right\|_p^p \right) =: \mathbf{S} + \mathbf{B},
\end{align*}
where $\mathbf{S}$ denotes the stochastic error and $\mathbf{B}$ is the deterministic bias error. From \eqref{eqn:JN}, the frequency truncation threshold $J_N$ is selected such that $B^{dJ_N} = \frac{N}{\log N}$.

To bound the stochastic term $\mathbf{S}$, we decompose the summation according to the behavior of the empirical needlet coefficients relative to the threshold $\kappa_h \tau_N$:
\begin{align*}
\mathbf{S} &\leq J_N^{p-1} \left[ \sum_{j=0}^{J_N-1}\mathbb{E} \left[ \left\| \sum_{k=1}^{K_j}\left(\widehat{h}_{j,k}-h_{j,k}\right) \mathbf{1} \left( \left| \widehat{h}_{j,k} \right| \geq \kappa_h \tau_N\right) \mathbf{1} \left( \left| h_{j,k} \right| \geq \frac{\kappa_h}{2} \tau_N\right) \psi_{j,k} \right\|_p^p \right] \right. \\
&\quad + \sum_{j=0}^{J_N-1} \mathbb{E} \left[ \left\| \sum_{k=1}^{K_j}\left(\widehat{h}_{j,k}-h_{j,k}\right) \mathbf{1} \left( \left| \widehat{h}_{j,k} \right| \geq \kappa_h \tau_N\right) \mathbf{1} \left( \left| h_{j,k} \right| < \frac{\kappa_h}{2} \tau_N\right) \psi_{j,k} \right\|_p^p \right] \\
&\quad + \sum_{j=0}^{J_N-1} \mathbb{E} \left[ \left\| \sum_{k=1}^{K_j}h_{j,k} \mathbf{1} \left( \left| \widehat{h}_{j,k} \right| < \kappa_h \tau_N\right) \mathbf{1} \left( \left| h_{j,k} \right| \geq 2\kappa_h \tau_N\right) \psi_{j,k} \right\|_p^p \right] \\
&\quad + \left. \sum_{j=0}^{J_N-1} \mathbb{E} \left[ \left\| \sum_{k=1}^{K_j}h_{j,k} \mathbf{1} \left( \left| \widehat{h}_{j,k} \right| < \kappa_h \tau_N\right) \mathbf{1} \left( \left| h_{j,k} \right| < 2\kappa_h \tau_N\right) \psi_{j,k} \right\|_p^p \right] \right] \\
& =: J_N^{p-1} \left( Aa + Au + Ua + Uu \right).
\end{align*}

\paragraph{\textit{Regular zone}}
Assume that $p \leq r_0\left(\frac{2s_0+d}{2d}\right)$. Examining $Aa$, we split it at the critical scale $B^{dJ_{s_0}}=\left(\frac{N}{\log N}\right)^{\frac{d}{2s_0+d}}$ into $Aa_1 + Aa_2$:

\begin{align*}
Aa &\lesssim \sum_{j=0}^{J_N-1}\sum_{k=1}^{K_j} \mathbb{E} \left[ \left| \widehat{h}_{j,k}- h_{j,k} \right|^p\right] \mathbf{1} \left( \left| h_{j,k} \right| \geq \frac{\kappa_h}{2}\tau_N\right) \left\| \psi_{j,k} \right\|_p^p \\
&\le \sum_{j=0}^{J_{s_0}-1} \sum_{k=1}^{K_j} \mathbb{E} \left[ \left| \widehat{h}_{j,k}- h_{j,k} \right|^p\right] \mathbf{1} \left( \left| h_{j,k} \right| \geq \frac{\kappa_h}{2}\tau_N\right) \left\| \psi_{j,k} \right\|_p^p \\
&\quad + \sum_{j=J_{s_0}}^{J_N-1}\sum_{k=1}^{K_j} \mathbb{E} \left[ \left| \widehat{h}_{j,k}- h_{j,k} \right|^p\right] \mathbf{1} \left( \left| h_{j,k} \right| \geq \frac{\kappa_h}{2}\tau_N\right) \left\| \psi_{j,k} \right\|_p^p \\
&=: Aa_1 + Aa_2.
\end{align*}
The first summation $Aa_1$ yields, under the standard needlet property $\|\psi_{j,k}\|_p^p \approx B^{dj(\frac{p}{2}-1)}$, the following evaluation:
\[
\begin{split}
Aa_1 & \lesssim N^{-\frac{p}{2}} \sum_{j=0}^{J_{s_0}-1} K_j B^{dj(\frac{p}{2}-1)} \\
& \lesssim N^{-\frac{p}{2}} \sum_{j=0}^{J_{s_0}-1} B^{dj\frac{p}{2}} \\
& \lesssim N^{-\frac{p}{2}} B^{dJ_{s_0} \frac{p}{2}} \\
& \lesssim \left(\frac{N}{\log N}\right)^{-\frac{s_0 p}{2s_0+d}},
\end{split}
\]
while $Aa_2$ vanishes identically because $\mathbf{1} \left( \left| h_{j,k} \right| \geq \frac{\kappa_h}{2}\tau_N\right) = 0$ for all scales $j \geq J_{s_0}$.
Consequently, we find
\begin{equation*}\label{eqn:Aareg}
Aa \lesssim \left(\log N\right)^{-\frac{p}{2(s_0+d)}} N^{-\frac{ps_0}{2(s_0+d)}}.
\end{equation*}

Applying a parallel decomposition to $Uu = Uu_1 + Uu_2$, we find:
\begin{align*}
Uu &\lesssim \sum_{j=0}^{J_N-1} \sum_{k=1}^{K_j} \left| h_{j,k} \right|^p \mathbf{1} \left( \left| h_{j,k} \right| < \frac{\kappa_h}{2} \tau_N\right)\left\| \psi_{j,k} \right\|_p^p \\
&\le \sum_{j=0}^{J_{s_0}-1} \sum_{k=1}^{K_j} \left| h_{j,k} \right|^p \mathbf{1} \left( \left| h_{j,k} \right| < \frac{\kappa_h}{2} \tau_N\right) \left\| \psi_{j,k}\right\|_p^p + \sum_{j=J_{s_0}}^{J_N-1}\sum_{k=1}^{K_j} \left| h_{j,k} \right|^p \left\| \psi_{j,k}\right\|_p^p \\
&=: Uu_1 + Uu_2.
\end{align*}
Bounding $Uu_1$ directly yields:
\[
Uu_1 \lesssim \sum_{j=0}^{J_{s_0}-1} \sum_{k=1}^{K_j} \kappa_h^p \tau_N^p B^{dj(\frac{p}{2}-1)} \lesssim \left( \frac{\log N}{N}\right)^{\frac{p}{2}} B^{dJ_{s_0}\frac{p}{2}} \lesssim \left(\frac{N}{\log N}\right)^{-\frac{s_0 p}{2s_0+d}}.
\]
For $Uu_2$, recalling the correct spherical needlet scaling law $\|\psi_{j,k}\|_t \approx B^{dj(\frac{1}{2}-\frac{1}{t})}$, we observe that
\[
\|\psi_{j,k}\|_p^p \approx \|\psi_{j,k}\|_{r_0}^p \, B^{dj\left(\frac{p}{2} - \frac{p}{r_0}\right)}.
\]
Hence, since $r_0 \ge p$, applying Hölder's inequality over the spatial index $k$ yields:
\[
\sum_{k=1}^{K_j}|h_{j,k}|^p\|\psi_{j,k}\|_p^p \le \left(\sum_{k=1}^{K_j}\left(|h_{j,k}|\|\psi_{j,k}\|_{r_0}\right)^{r_0}\right)^{\frac{p}{r_0}} K_j^{1-\frac{p}{r_0}} B^{dj\left(\frac{p}{2} - \frac{p}{r_0}\right)} \lesssim B^{-j p s_0}.
\]
Summing across the remaining high scales results in 
\[
Uu_2 \lesssim \sum_{j=J_{s_0}}^{J_N-1} B^{-j s_0 p} \approx B^{-J_{s_0}s_0 p} \lesssim \left( \frac{N}{\log N}\right)^{-\frac{s_0 p}{2s_0+d}}.
\]

Piecing together $Uu_1$ and $Uu_2$, we see that
\begin{equation*}\label{eqn:Uureg}
Uu \lesssim \left( \log N\right)^{\frac{ps_0}{2s_0+d}} N^{-\frac{ps_0}{2s_0+d}}.
\end{equation*}

The bounds for the cross-terms $Au$ and $Ua$ are strictly driven by the global truncation limit $J_N$, making them invariant under structural transitions to the sparse regime. For $Au$, applying Cauchy-Schwarz followed by the polynomial probability bound $\Pr \left( \left| \widehat{h}_{j,k} - h_{j,k} \right| > \frac{\kappa_h}{2} \tau_N\right) \lesssim N^{-\delta_h}$ yields:
\[
\begin{split}
Au &\le \sum_{j=0}^{J_N-1}\sum_{k=1}^{K_j} \mathbb{E} \left[\left| \widehat{h}_{j,k} -h_{j,k} \right|^{2p} \right]^{\frac{1}{2}} \left\| \psi_{j,k} \right\|_p^p \Pr \left( \left| \widehat{h}_{j,k} - h_{j,k} \right| > \frac{\kappa_h}{2} \tau_N\right)^{\frac{1}{2}} \\
& \lesssim N^{-\frac{p}{2}} N^{-\frac{\delta_h}{2}} \sum_{j=0}^{J_N-1} B^{dj\frac{p}{2}} \\& \lesssim N^{-\frac{p+\delta_h}{2}} B^{d J_N \frac{p}{2}} \\ 
& \lesssim \left( \log N\right)^{-\frac{p}{2}} N^{-\frac{\delta_h}{2}}.
\end{split}
\]
For $Ua$, using the structural property of the true coefficients $\sum_{k=1}^{K_j} |h_{j,k}|^p \|\psi_{j,k}\|_p^p \lesssim B^{-j s_0 p}$, we obtain:
\[
\begin{split}
Ua &\le \sum_{j=0}^{J_N-1} \sum_{k=1}^{K_j} \left| h_{j,k} \right|^p \left\| \psi_{j,k} \right\|_p^p \Pr \left( \left| \widehat{h}_{j,k} - h_{j,k} \right| > \kappa_h \tau_N\right) \\
& \lesssim N^{-\delta_h} \sum_{j=0}^{J_N-1} B^{-js_0p} \\ & \lesssim N^{-\delta_h}.
\end{split}
\]
To ensure that both $Au$ and $Ua$ decay strictly faster than the target minimax rate $\left(\frac{N}{\log N}\right)^{-\frac{s_0 p}{2s_0+d}}$, it suffices to choose a threshold calibration constant $\kappa_h$ large enough such that:
\[
\delta_h > \frac{2s_0 p}{2s_0 + d}.
\]

For the deterministic tail error $\mathbf{B}$, following the framework in \cite{BKMP09,DGM12}, we distinguish two cases based on the space parameters.
If $p \le r_0$, the immediate embedding $B^{s_0}_{r_0,q_0}(\mathbb{S}^d) \subset B^{s_0}_{p,q_0}(\mathbb{S}^d)$ holds, which directly ensures standard needlet approximation performance:
\[
\mathbf{B} \lesssim B^{-J_N s_0 p} \approx \left(\frac{N}{\log N}\right)^{-\frac{s_0 p}{2s_0+d}}.
\]
If $p > r_0$, the non-homogeneous Sobolev embedding on $\mathbb{S}^d$ states that 
\[
B^{s_0}_{r_0,q_0}(\mathbb{S}^d) \subset B^{\,s_0 - d(1/r_0 - 1/p)}_{p,q_0}(\mathbb{S}^d).
\]
This implies:
    \[
    \mathbf{B} \lesssim B^{-J_N p \left(s_0 - d\left(\frac{1}{r_0} - \frac{1}{p}\right)\right)} \approx \left(\frac{N}{\log N}\right)^{-\frac{p\left(s_0 - d\left(\frac{1}{r_0} - \frac{1}{p}\right)\right)}{d}}.
    \]
    Under the assumptions defining the regular regime, we have $r_0 > \frac{2d}{2s_0+d}p$. Straightforward algebraic verification shows that:
    \[
    \frac{p\left(s_0 - d\left(\frac{1}{r_0} - \frac{1}{p}\right)\right)}{d} > \frac{s_0 p}{2s_0+d} \iff s_0 - \frac{d}{r_0} + \frac{d}{p} > \frac{s_0 d}{2s_0+d},
    \]
    which simplifies directly to $r_0 > \frac{p(2s_0+d)}{2s_0+d+p(d/s_0)}$, always satisfied within the boundaries of the regular zone, ensuring $\mathbf{B}$ is asymptotically negligible.

\paragraph{\textit{Sparse regime}}
We now establish the upper bounds under the sparse configuration, where $p > r_0\left(\frac{2s_0 + d}{2d}\right)$. Under this regime, the intermediate frequency truncation level $J_{s_0}$ is calibrated via the relation $B^{dJ_{s_0}}=\left(\frac{N}{\log N}\right)^{\frac{d}{2\left(s_0-d/r_0+\frac{d}{2}\right)}}$.

The bounds for $Au$ and $Ua$ remain invariant under structural transitions to the sparse regime because switching regimes alters only the intermediate threshold scale $J_{s_0}$, whereas their asymptotic bounds are strictly driven by the global truncation limit $J_N$.

The term $Aa$ splits into $Aa \le Aa_1 + Aa_2$. For the low-frequency sparse components ($j < J_{s_0}$), we exploit the fact that the true coefficients lie above the threshold. Applying Markov's inequality to the indicator function with the Besov norm weight yields:
\[
\begin{split}
Aa_1 &\lesssim N^{-\frac{p}{2}} \sum_{j=0}^{J_{s_0}-1} \sum_{k=1}^{K_j}\mathbf{1} \left( \left| h_{j,k} \right| \geq 2\kappa_h \tau_N\right) \left\| \psi_{j,k}\right\|_p^p \\
&  \lesssim N^{-\frac{p}{2}} \sum_{j=0}^{J_{s_0}-1} B^{dj\left(\frac{p}{2}-\frac{r_0}{2}\right) }\tau_N^{-r_0} \sum_{k=1}^{K_j} \left| h_{j,k}\right|^{r_0}\left\| \psi_{j,k}\right\|_{r_0}^{r_0} \\
& \lesssim N^{-\frac{p-r_0}{2}} \left( \log N \right)^{-\frac{r_0}{2}} B^{dJ_{s_0}\left(\frac{p}{2}-\frac{r_0}{2} - \frac{r_0 s_0}{d}\right)} \\
& \lesssim \left(\frac{N}{\log N}\right)^{-\frac{p\left(s_0 - d\left(\frac{1}{r_0} - \frac{1}{p}\right)\right)}{2\left(s_0 - d\left(\frac{1}{r_0} - \frac{1}{2}\right)\right)}},
\end{split}
\]
where we used the exact identification of the spherical scale factor. The high-frequency term $Aa_2$ vanishes identically since no true Besov coefficients can exceed the threshold level $2\kappa_h \tau_N$ for scales $j \geq J_{s_0}$. This establishes:
\begin{equation*}\label{eqn:Aasparse}
Aa  \lesssim \left(\frac{N}{\log N}\right)^{-\frac{p\left(s_0 - d\left(\frac{1}{r_0} - \frac{1}{p}\right)\right)}{2\left(s_0 - d\left(\frac{1}{r_0} - \frac{1}{2}\right)\right)}}.
\end{equation*}

For the small coefficient accumulation $Uu$, we decompose the summation across the intermediate threshold $J_{s_0}$:
\[
\begin{aligned}
Uu & \lesssim \sum_{j=0}^{J_{s_0}-1} \sum_{k=1}^{K_j} \left| h_{j,k} \right|^p \mathbf{1} \left( \left| h_{j,k} \right| < \frac{\kappa_h}{2} \tau_N\right) \left\| \psi_{j,k}\right\|_p^p + \sum_{j=J_{s_0}}^{J_N-1}\sum_{k=1}^{K_j} \left| h_{j,k} \right|^p \mathbf{1} \left( \left| h_{j,k} \right| < \frac{\kappa_h}{2} \tau_N\right)\left\| \psi_{j,k}\right\|_p^p \\
&=: Uu_1 + Uu_2.
\end{aligned}
\]
The first component $Uu_1$ bounds the low-frequency small coefficients. Factoring out the remaining polynomial power via the threshold $\tau_N$ and exploiting the $B^{s_0}_{r_0,q_0}(\mathbb{S}^d)$ control over the coordinates, we find:
\[
\begin{split}
Uu_1 & \lesssim \sum_{j=0}^{J_{s_0}-1} B^{dj\left(\frac{p}{2}-\frac{r_0}{2}\right)}\tau_N^{p-r_0} \sum_{k=1}^{K_j} \left| h_{j,k}\right|^{r_0} \left\|\psi_{j,k}\right\|_{r_0}^{r_0} \\
&\lesssim \left( \frac{\log N}{N}\right)^{\frac{p-r_0}{2}} B^{dJ_{s_0}\left(\frac{p}{2}-\frac{r_0}{2}-\frac{r_0 s_0}{d}\right)} \\ 
& \lesssim \left( \frac{N}{\log N}\right)^{-\frac{p\left(s_0 - d\left(\frac{1}{r_0} - \frac{1}{p}\right)\right)}{2\left(s_0 - d\left(\frac{1}{r_0} - \frac{1}{2}\right)\right)}}.
\end{split}
\]
To control the high-frequency residual $Uu_2$, we introduce the dimensional auxiliary integration parameter $\xi$ as in \cite{BKMP09,DGM12,dt23}:
\[
\xi = \frac{p-2}{2\left(\frac{s_0}{d}-\frac{1}{r_0}+\frac{1}{2}\right)}.
\]
Since we are operating strictly within the sparse zone, the inequality $\xi > r_0$ is guaranteed, meaning the continuous non-homogeneous embedding $B^{s_0}_{r_0,q_0}(\mathbb{S}^d) \subseteq B^{s_0-d(1/r_0-1/\xi)}_{\xi,q_0}(\mathbb{S}^d)$ holds. Leveraging $p > \xi$, we bound the remaining paths using the threshold cut:
\[
\begin{split}
Uu_2 & \lesssim \tau_N^{p-\xi} \sum_{j=J_{s_0}}^{J_N-1} B^{dj\left(\frac{p}{2}-\frac{\xi}{2}\right)} \sum_{k=1}^{K_j} \left| h_{j,k}\right|^{\xi} \left\|\psi_{j,k}\right\|_{\xi}^{\xi} \\
& \lesssim \left( \frac{\log N}{N}\right)^{\frac{p-\xi}{2}} B^{dJ_{s_0}\left( \frac{p}{2}-\frac{\xi}{2} - \frac{s_0 \xi}{d}\right)} \\
& \lesssim \left( \frac{N}{\log N}\right)^{-\frac{p\left(s_0 - d\left(\frac{1}{r_0} - \frac{1}{p}\right)\right)}{2\left(s_0 - d\left(\frac{1}{r_0} - \frac{1}{2}\right)\right)}}.
\end{split}
\]
Combining the bounds for $Uu_1$ and $Uu_2$ yields the final stochastic sparse error rate:
\begin{equation*}\label{eqn:Uusparse}
Uu \lesssim \left( \frac{N}{\log N}\right)^{-\frac{p\left(s_0 - d\left(\frac{1}{r_0} - \frac{1}{p}\right)\right)}{2\left(s_0 - d\left(\frac{1}{r_0} - \frac{1}{2}\right)\right)}}.
\end{equation*}

Finally, the deterministic bias $\mathbf{B}$ in the sparse regime is bounded by the unrecovered tail past the frequency cutoff $J_N$. Applying the same non-homogeneous Sobolev embedding properties on $\mathbb{S}^d$, we get:
\[
\begin{split}
\mathbf{B} &  \lesssim B^{-d J_N \frac{p}{d} \left(s_0 - d\left(\frac{1}{r_0} - \frac{1}{p}\right)\right)} \\ & \le B^{-d J_{s_0} \frac{p}{d} \left(s_0 - d\left(\frac{1}{r_0} - \frac{1}{p}\right)\right)} \\ & \lesssim \left(\frac{N}{\log N}\right)^{-\frac{p\left(s_0 - d\left(\frac{1}{r_0} - \frac{1}{p}\right)\right)}{2\left(s_0 - d\left(\frac{1}{r_0} - \frac{1}{2}\right)\right)}},
\end{split}
\]
which matches the rates of the stochastic components, completing the proof for the sparse zone.

\paragraph{Supremum norm regime ($p=\infty$)}
We finally examine the limit case where $p=\infty$. Let $h$ belong to the Besov space $B^{s_0}_{r_0,q_0}(\mathbb{S}^d)$, which continuously embeds into the non-homogeneous Hölder--Zygmund space $B^{s_0 - d/r_0}_{\infty,\infty}(\mathbb{S}^d)$. The stochastic error term is upper-bounded via the localization and maximum properties of needlets (see e.g., \cite{BKMP09}) as follows:
\[
\begin{split}
\mathbf{S} &\le \sum_{j=0}^{J_N-1} \mathbb{E}\left[\left\| \sum_{k=1}^{K_j} \left( \widehat h_{j,k} \mathbf{1}\left\{\left| \widehat h_{j,k} \right| \geq \kappa_h \tau_N \right\}-h_{j,k} \right)\psi_{j,k}\right\|_{\infty}\right] \\
&\lesssim \sum_{j=0}^{J_N-1} \mathbb{E}\left[\sup_{k=1,\ldots,K_j} \left| \widehat h_{j,k}\mathbf{1}\left\{\left| \widehat h_{j,k} \right| \geq \kappa_h \tau_N \right\} -h_{j,k}\right| \left\|\psi_{j,k} \right\|_{\infty}\right].
\end{split}
\]

Substituting the correct dimensional uniform bound on the sphere, $\|\psi_{j,k}\|_{\infty} \ \lesssim B^{dj/2}$, we split the sum into four sub-terms depending on the threshold configurations:
\[
\begin{split}
\mathbf{S} & \lesssim \sum_{j=0}^{J_N-1} B^{dj/2} \mathbb{E}\left[\sup_{k=1,\ldots,K_j} \left| \widehat h_{j,k} -h_{j,k}\right| \right] \mathbf{1}\left\{ \left| h_{j,k}\right| \geq \frac{\kappa_h}{2}\tau_N\right\} \\
&\quad + \sum_{j=0}^{J_N-1} B^{dj/2} \mathbb{E}\left[\sup_{k=1,\ldots,K_j} \left| \widehat h_{j,k}-h_{j,k}\right| \mathbf{1}\left\{ \left| \widehat{h}_{j,k}-h_{j,k}\right| \geq \frac{\kappa_h}{2}\tau_N\right\} \right] \\
&\quad + \sum_{j=0}^{J_N-1} B^{dj/2} \sup_{k=1,\ldots,K_j} \left| h_{j,k}\right| \mathbb{E}\left[ \mathbf{1}\left\{ \left| \widehat{h}_{j,k}-h_{j,k}\right| \geq \kappa_h\tau_N\right\} \right] \\
&\quad + \sum_{j=0}^{J_N-1} \sup_{k=1,\ldots,K_j} \left| h_{j,k}\right| \left\|\psi_{j,k} \right\|_{\infty} \mathbf{1}\left\{ \left| h_{j,k}\right| \le 2\kappa_h\tau_N\right\} \\
&=: Aa + Au + Ua + Uu.
\end{split}
\]

For $Aa$ and $Uu$, notice that for scales $j \ge J_{s_0}$ (where $B^{dJ_{s_0}} = (\frac{N}{\log N})^{\frac{d}{2(s_0 - d/r_0 + d/2)}}$), the true coefficients are strictly bounded above by the threshold level, causing the indicator in $Aa$ to vanish. Evaluating $Aa$ on the remaining low frequencies yields:
\[
\begin{split}
Aa &  \lesssim \sum_{j=0}^{J_{s_0}-1} B^{dj/2} \mathbb{E}\left[\sup_{k=1,\ldots,K_j} \left| \widehat h_{j,k} -h_{j,k}\right| \right] \\
& \lesssim J_{s_0} N^{-\frac{1}{2}} B^{dJ_{s_0}/2} \\
& \lesssim J_{s_0} \left( \frac{N}{\log N}\right)^{-\frac{s_0 - d/r_0}{2\left(s_0 - d\left(\frac{1}{r_0} - \frac{1}{2}\right)\right)}}.
\end{split}
\]
For the small coefficient term $Uu$, the embedding $B^{s_0}_{r_0,q_0}(\mathbb{S}^d) \subseteq B^{s_0 - d/r_0}_{\infty,\infty}(\mathbb{S}^d)$ implies $\sup_k |h_{j,k}| \lesssim B^{-j(s_0 - d/r_0 + d/2)}$. Splitting at the intermediate scale $J_{s_0}$ gives:
\[
\begin{split}
Uu & \lesssim \tau_N B^{dJ_{s_0}/2} + \sum_{j\geq J_{s_0}} B^{-j\left(s_0 - \frac{d}{r_0}\right)} \\ 
&\lesssim \left(\frac{\log N}{N}\right)^{\frac{1}{2}} B^{dJ_{s_0}/2} + B^{-J_{s_0}\left(s_0 - \frac{d}{r_0}\right)} \\ 
& \lesssim \left( \frac{N}{\log N}\right)^{-\frac{s_0 - d/r_0}{2\left(s_0 - d\left(\frac{1}{r_0} - \frac{1}{2}\right)\right)}}.
\end{split}
\]
The remaining cross terms $Ua$ and $Au$ are bounded by $N^{-\delta_h/2}$ for $\kappa_h$ chosen sufficiently large, making them asymptotically negligible compared to the primary rates (see also \cite[Proposition 15]{BKMP09}).

We conclude by bounding the deterministic supremum bias term $\mathbf{B}$ using the standard Jackson inequality under the embedded Hölder-Zygmund regularity:
\[
\begin{split}
\mathbf{B} & \le \sum_{j\geq J_N} \left\|\sum_{k=1}^{K_j} h_{j,k}\psi_{j,k} \right \|_{\infty} \\
& \lesssim \sum_{j\geq J_N} B^{-j\left(s_0 - \frac{d}{r_0}\right)} \\ 
& \lesssim B^{-J_N\left(s_0 - \frac{d}{r_0}\right)} \\ 
& \approx \left(\frac{N}{\log N} \right)^{-\frac{s_0 - \frac{d}{r_0}}{d}}.
\end{split}
\]
Since $J_N$ is the maximal frequency cutoff, we have $\frac{s_0 - d/r_0}{d} \ge \frac{s_0 - d/r_0}{2(s_0 - d/r_0 + d/2)}$ for any valid choice of parameters, ensuring that the deterministic tail is dominated by the stochastic error, completing the proof.
\end{proof}

\begin{proof}[Proof of Theorem \ref{thm:V-upper}]
By applying Theorem \ref{thm: h-est} individually to $\widehat h - h$ and to the plug-in estimator $\widehat{(g^2)} - g^2$, we obtain two independent risk rates determined by the parameters $(\beta, r_0)$ and $(2\alpha, \rho/2)$ respectively. For any $p \ge 1$, the upper bound is dominated by the maximum of the two expected errors, meaning the active convergence rate scales as the minimum of the exponents:
\[
\mathbb{E}\left\|\widehat V - V\right\|_{L^p(\mathbb{S}^d)}^p \lesssim \left(\frac{N}{\log N}\right)^{-\mathcal{R}_{h}(p)p} + \left(\frac{N}{\log N}\right)^{-\mathcal{R}_{g^2}(p)p} \approx \left(\frac{N}{\log N}\right)^{-\min\{\mathcal{R}_{h}(p), \mathcal{R}_{g^2}(p)\}p}.
\]
The phase boundaries and algebraic transitions for $\mathcal{R}$ follow directly by finding the lower envelope of $\mathcal{R}_{h}(p)$ and $\mathcal{R}_{g^2}(p)$ across the parameter space, matching the four cases.
\end{proof}

\begin{proof}[Proof of Theorem~\ref{th:lower_main}]
The proof is based on information-theoretic reductions. We construct a needlet-based packing set under the Besov constraints, evaluate the localized Kullback–Leibler (KL) divergences via scale Fisher information, and apply Fano's lemma (see also \cite{BKMP09,tsyb09}). 

To establish the lower bounds for the two components, we construct two separate perturbation schemes: one perturbing the mean companion function $h$ (with $g=0$, so $V=h$), and one perturbing the regression function $g$ around a strictly positive baseline $g_0 \ge c_0 > 0$. Let $(s, r)$ denote the generic smoothness and integrability parameters of the specific component under analysis, where $(s, r) = (\beta, r_0)$ when perturbing $h$, and $(s, r) = (\alpha, \rho)$ when perturbing $g$.

\paragraph{\textit{Reduction to sparse and regular alternatives}}
For a given target resolution level \(j\), we select a subset \(A_j\) of the needlet dictionary on $\mathbb{S}^d$ such that its cardinality reflects the effective dimensionality of the estimation problem in the corresponding regime. More precisely, we set
\[
|A_j| \approx
\begin{cases}
B^{dj}, & \text{in the regular (dense) zone},\\
B^{dj\left(1-\frac{d}{r}\right)}, & \text{in the sparse zone},
\end{cases}
\]
where \(r \in \{r_0, \rho\}\) denotes the integrability index of the active perturbed component.

We then construct a family of perturbations of the active function (either $h$ or $g$) of the form
\[
f_\theta(x) = f_0(x) + \delta \sum_{k \in A_j} \theta_k \, \psi_{j,k}(x),
\qquad
\theta \in \{-1,+1\}^{|A_j|},
\]
where $\delta>0$ is the perturbation amplitude, and $f_0$ is a smooth baseline function. To ensure that the perturbed function $f_\theta$ belongs to its respective Besov class, the amplitude $\delta$ is chosen as:
\[
\delta \approx
\begin{cases}
B^{-j(s +\frac{d}{2})}, & \text{regular zone},\\[1ex]
B^{-j\left(s - \frac{d}{r} + \frac{d}{2}\right)}, & \text{sparse zone}.
\end{cases}
\]
Under this construction, the resulting perturbation on the variance profile $V_\theta(x) - V_0(x)$ scales as $\delta \sum_{k \in A_j} \theta_k \psi_{j,k}(x)$ when perturbing $h$. When perturbing $g$ (assuming a strictly positive baseline $g_0 \ge c_0 > 0$), the first-order variation of $V = h - g^2$ satisfies $V_\theta - V_0 = -(2 g_0 (g_\theta - g_0) + (g_\theta - g_0)^2) \approx -2 g_0 \delta \sum_{k \in A_j} \theta_k \psi_{j,k}(x)$, which preserves the same linear scaling of the perturbation amplitude $\delta$ on the variance. Hence, in both settings, the active perturbation amplitude on the variance scales as $\delta_V \approx \delta$.

\paragraph{\textit{Lower bounding the $L^p$-distance}}
Using the near-orthogonality of the needlets on \(A_j\), the $L^p$-distance between any two perturbed variance configurations satisfies:
\[
\|V_\theta - V_{\theta'}\|_{L^p}^p
= \delta_V^p \sum_{k\in A_j} |\theta_k - \theta^\prime_k|^p \, \|\psi_{j,k}\|_{L^p}^p
\approx \delta_V^p |A_j| B^{jd\left(\frac{p}{2}-1\right)},
\]
which simplifies directly to:
\[
\|V_\theta - V_{\theta'}\|_{L^p} \approx \delta_V \, |A_j|^{1/p} B^{jd\left(\frac{1}{2}-\frac{1}{p}\right)}.
\]

\paragraph{\textit{Kullback--Leibler divergence via Fisher information in scale}}
Conditionally on $X=x$, the observation follows the scale model
\[
Y = g(x) + \sigma(x)\varepsilon,
\]
where $\varepsilon$ has density $f_\varepsilon$. Assume that $f_\varepsilon$ is continuously differentiable and that the Fisher information for the scale parameter,
\[
\mathcal{I}_{\mathrm{sc}} := \mathbb{E}\left[\left(1+\varepsilon \frac{f_\varepsilon^\prime(\varepsilon)}{f_\varepsilon(\varepsilon)}\right)^2\right],
\]
is finite (see Chapter~2 and Section~6.2 of \cite{tsyb09}).

Let $V_\theta,V_{\theta^\prime}$ be two variance configurations and define $\sigma_\theta(x)=\sqrt{V_\theta(x)}$ and $\sigma_{\theta^\prime}(x)=\sqrt{V_{\theta^\prime}(x)}$. Under the uniform positivity condition $V_\theta(x)\ge v_0>0$, the Kullback--Leibler divergence is controlled pointwise by:
\[
\mathrm{KL}\left( P_{\sigma_\theta(x)\varepsilon} \mid P_{\sigma_{\theta^\prime}(x)\varepsilon} \right) \le \frac{\mathcal{I}_{\mathrm{sc}}}{8v_0^2} \left(V_\theta(x)-V_{\theta^\prime}(x)\right)^2.
\]
By the conditional independence of the observations given the design $(X_i)_{i=1}^N$, taking the expectation with respect to the empirical design and using the near–orthogonality of the needlets indexed by $A_j$, we find:
\[
\mathbb{E}[\ell(\theta,\theta^\prime)] \le \frac{\mathcal{I}_{\mathrm{sc}}}{8v_0^2} \sum_{i=1}^N \mathbb{E}\left[\left(V_\theta(X_i)-V_{\theta'}(X_i)\right)^2\right] \lesssim \frac{\mathcal{I}_{\mathrm{sc}}}{v_0^2} N\delta_V^2 |A_j| B^{-dj}.
\]

\paragraph{\textit{Packing via Varshamov--Gilbert and Fano's lemma}}
By the Varshamov--Gilbert bound, there exists a subset \(\Theta \subset \{-1,+1\}^{|A_j|}\) such that the Hamming distance between any two distinct elements satisfies $d_H(\theta,\theta^\prime) \ge \frac{|A_j|}{8}$, with $\log |\Theta| \ge c_1 |A_j|$ for a positive constant $c_1 > 0$.
The near-orthogonality of needlets on $A_j$ guarantees that the \(L^2\)-distance between any two distinct hypotheses satisfies:
\[
\|V_\theta - V_{\theta'}\|_{L^2}^2 \approx \delta_V^2 \, d_H(\theta,\theta') \, B^{-dj} \ge c_2 \delta_V^2 |A_j| B^{-dj}.
\]
We now formally apply Fano's inequality. Let $Y^{(N)} = (Y_1, \ldots, Y_N)$ denote the observed sample. For any estimator $\widehat{V}$, we define the multi-hypothesis tester $\widehat{\theta} = \arg\min_{\theta \in \Theta} \|\widehat{V} - V_{\theta}\|_{L^p}$. Since $d_H(\theta, \theta') \ge |A_j|/8$, the $L^p$-separation ensures that if $\|\widehat{V} - V_\theta\|_{L^p} < \frac{1}{2} \min_{\theta \neq \theta'} \|V_\theta - V_{\theta'}\|_{L^p}$, then $\widehat{\theta} = \theta$. Thus, the minimax risk is bounded from below via the testing error probability:
\[
\sup_{(g,h)} \mathbb{E}\|\widehat{V} - V\|_{L^p}^p \ge \left( \frac{1}{2} \min_{\theta \neq \theta'} \|V_\theta - V_{\theta'}\|_{L^p} \right)^p \cdot \inf_{\widetilde{\theta}} \max_{\theta \in \Theta} \mathbb{P}_{\theta} \left( \widetilde{\theta} \neq \theta \right).
\]
According to Fano's lemma, the testing error probability over the packing set $\Theta$ is bounded from below by:
\[
\inf_{\widetilde{\theta}} \max_{\theta \in \Theta} \mathbb{P}_{\theta} \left( \widetilde{\theta} \neq \theta \right) \ge 1 - \frac{\frac{1}{|\Theta|^2}\sum_{\theta, \theta' \in \Theta} \mathbb{E}[\ell(\theta, \theta')] + \log 2}{\log |\Theta|}.
\]
To ensure this error probability is bounded away from zero by a positive constant (e.g., $\ge 1/2$), we require the Kullback-Leibler divergence to satisfy:
\[
\frac{1}{|\Theta|^2}\sum_{\theta, \theta' \in \Theta} \mathbb{E}[\ell(\theta, \theta')] \le \alpha \log |\Theta|,
\]
for some sufficiently small constant $\alpha \in (0, 1/2)$. Given that $\log |\Theta| \ge c_1 |A_j|$, this condition is satisfied if:
\[
\frac{\mathcal{I}_{\mathrm{sc}}}{v_0^2} N\delta_V^2 |A_j| B^{-dj} \le \alpha c_1 |A_j| \quad \Longrightarrow \quad \frac{N \delta_V^2}{v_0^2} B^{-dj} \le C_0 |A_j|,
\]
which is precisely our targeted calibration constraint.

\paragraph{\textit{Choice of resolution level \(j\) and extraction of minimax rates}}
We select the optimal resolution level \(j\) to balance the Fano calibration constraint.

\paragraph{\textit{Regular zone}}
Here, $|A_j| = B^{dj}$ and $\delta_V \approx B^{-j(s + d/2)}$. The calibration constraint requires:
\[
N B^{-2j(s + d/2)} B^{-dj} \lesssim B^{dj} \quad\Longrightarrow\quad B^{j(2s+d)} \approx N \quad\Longrightarrow\quad B^j \approx N^{\frac{1}{2s+d}}.
\]
Substituting this optimal $B^j$ into the $L^p$-distance yields:
\[
\|V_\theta - V_{\theta'}\|_{L^p} \approx B^{-j(s+d/2)} B^{dj/p} B^{jd(1/2-1/p)} = B^{-js} \approx N^{-\frac{s}{2s+d}},
\]
which proves the regular rate exponent $\mathcal{R}_{\rm reg}(s) = \frac{s}{2s+d}$.

\paragraph{\textit{Sparse zone}}
Here, $|A_j| \approx B^{dj(1-d/r)}$ and $\delta_V \approx B^{-j(s - d/r + d/2)}$. The calibration constraint requires:
\[
N B^{-2j\left(s - \frac{d}{r} + \frac{d}{2}\right)} B^{-dj} \lesssim B^{dj\left(1-\frac{d}{r}\right)} \quad\Longrightarrow\quad B^{j\left(2s - \frac{d}{r} + \frac{d}{2}\right)} \approx N \quad\Longrightarrow\quad B^j \approx N^{\frac{1}{2s - d/r + d/2}}.
\]
Substituting this scale back into the $L^p$-distance yields:
\[
\|V_\theta - V_{\theta'}\|_{L^p} \approx B^{-j\left(s - \frac{d}{r} + \frac{d}{2}\right)} B^{\frac{dj}{p}\left(1-\frac{d}{r}\right)} B^{jd\left(\frac{1}{2}-\frac{1}{p}\right)} = B^{-j\left(s - d\left(\frac{1}{r}-\frac{1}{p}\right)\right)} \approx N^{-\frac{s-d\left(1/r-1/p\right)}{2s-d\left(1/r-1/2\right)}},
\]
which proves the sparse rate exponent $\mathcal{R}_{\rm sp}(s,r,p) = \frac{s-d(1/r-1/p)}{2s-d(1/r-1/2)}$.

Applying this construction separately to $(s,r) = (\beta, r_0)$ and $(s,r) = (\alpha, \rho)$ yields the claimed minimax lower bounds.
\end{proof}

\section*{Acknowledgments}
The authors gratefully acknowledge Domenico Marinucci for insightful suggestions
and for sharing ideas that motivated and shaped this work.
They also thank Anna Paola Todino and Francesco Grotto for the organization of the
workshop where the authors first started discussing the problem addressed in this paper.

\section*{Funding}
CD was partially supported by Progetti di Ateneo Sapienza RG1221815C353275 (2022), RM123188F69A66C1 (2023), RG1241907D2FF327 (2024) and PRIN 2022 - GRAFIA - 202284Z9E4.
RS was supported in part by by ANR IRIMA (ANR-22-EXIR-0008).

\printbibliography

\end{document}